\documentclass[12pt,dvipsnames]{amsart}
\usepackage{amsmath,amsthm,mathtools,amsfonts,amssymb,tikz,verbatim,curves,enumitem, bbm, bm, mathrsfs, float, fullpage, xfrac, hyperref,graphicx,nicematrix,relsize}
\usepackage{xcolor}
\usepackage{mathtools}

\usepackage[final]{microtype}
\usepackage[retainorgcmds]{IEEEtrantools}
\usepackage{pdflscape}
\usetikzlibrary{arrows,automata,positioning,shapes,shadows}
\usepackage{enumitem}
\usepackage[style=numeric,backend=bibtex,maxbibnames=100]{biblatex} 
\usepackage[pagewise]{lineno}

\newtheorem{theorem}{Theorem}[section]
\newtheorem{prop}[theorem]{Proposition}
\newtheorem{lemma}[theorem]{Lemma}
\newtheorem{cor}[theorem]{Corollary}

\DeclareMathOperator{\lcm}{lcm}

\theoremstyle{definition}

\newtheorem{rem}[theorem]{Remark}

\theoremstyle{definition}
\newtheorem{predef}[theorem]{Definition}

\theoremstyle{definition}

\theoremstyle{remark}

\theoremstyle{remark}
\newtheorem{example}[theorem]{Example}

\newcommand{\End}{\mathop{\mathrm{End}}\nolimits}

\newcommand{\core}{\mathop{\mathrm{core}}\nolimits}

\newcommand{\xn}{X_{n}}
\newcommand{\xnp}{\xn^{+}}
\newcommand{\xnn}{X_n^{\mathbb{N}}}
\newcommand{\xnN}{X_n^{-\mathbb{N}}}

\newcommand{\spn}[1]{\widetilde{\mathscr{p}_{#1}}}

\newcommand{\shn}[1]{\widetilde{\mathcal{H}_{#1}}}
\newcommand{\hn}[1]{\mathcal{H}_{#1}}

\newcommand{\gen}[1]{\langle #1 \rangle}
\newcommand{\aut}[1]{\mbox{Aut}({#1})}

\newcommand{\shift}[1]{\sigma_{#1}}

\newcommand{\im}[1]{\mbox{Im}{(#1)}}

\newcommand{\Z}{\mathbb{Z}}
\newcommand{\N}{\mathbb{N}}
\newcommand{\spnprod}[1]{\ast_{\spn{n}}}

\newcommand{\dual}[1]{{#1}^{\vee}}
\newcommand{\dpi}{\dual{\pi}}
\newcommand{\disc}[1]{\mathop{\mathrm{disc}}(#1)}
\newcommand{\discT}[2]{\mathop{\mathrm{disc}}{\,\!\!}_{#2}(#1)}
\newcommand{\dlambda}{\dual{\lambda}}
\newcommand{\Duak}[2]{\mathscr{A}(\dual{#1}_{#2})}
\newcommand{\duA}[1]{\Duak{#1}{k}}

\newcommand{\rdual}[1]{\overline{\dual{#1}_{k}}}

\newcommand{\mb}[1]{\scalebox{1.2}{\ensuremath{#1}}}

\renewcommand{\bar}[1]{\overline{#1}}
\newcommand{\seteq}{:=}

\newcommand{\autskn}[1]{\aut{X_{#1}^{\N}, \shift{#1}}}

\newcommand{\asnn}{\autskn{n}}
\newcommand{\asnN}{\aut{X_{n}^{-\N}, \shift{n}}}

\newcommand{\letters}{\mathop{\mathrm{Letters}}}
\newcommand{\edges}{\mathop{\mathrm{E}}}
\newcommand{\lets}[2]{\letters(#1,#2)}
\newcommand{\letCnt}[2]{\mb{|}\lets{#1}{#2}\mb{|}}
\newcommand{\letsT}[3]{\letters{\,\!\!}_{#1}(#2,#3)}
\newcommand{\edgesT}[3]{\edges{\,\!\!}_{#1}(#2,#3)}



\author[C.~Bleak]{Collin Bleak}
\address{University of St Andrews, St Andrews, Scotland, UK}
\email{collin.bleak@st-andrews.ac.uk}

\author[F.~Olukoya]{Feyishayo Olukoya}
\address{University of St Andrews, St Andrews, Scotland, UK}
\email{fo55@st-andrews.ac.uk}

\begin{document}

\title{Conjugacy for certain automorphisms of the one-sided shift via transducers}

\date{}

\begin{abstract}
    We address the following open problem, implicit in the 1990 article \textit{Automorphisms of one-sided subshifts of finite type} of Boyle, Franks and Kitchens (BFK): 
    \begin{quote}
        Does there exists an element $\psi$ in the group of automorphisms of the one-sided shift Aut($\{0,1,\ldots,n-1\}^{\mathbb{N}}, \shift{n}$) so that all points of $\{0,1,\ldots,n-1\}^{\N}$ have orbits of length $n$ under $\psi$ and $\psi$ is not conjugate to a permutation?
    \end{quote}
    Here, by a \emph{permutation} we mean an automorphism of the one-sided shift dynamical system induced by a permutation of the symbol set $\{0,1,\ldots,n-1\}$.
    
    We resolve this question by showing, constructively, that any $\psi$ with properties as above must be conjugate to a permutation.

    Our techniques naturally extend those of BFK using the strongly synchronizing automata technology developed here and in several articles of the authors and collaborators (although, this article has been written to be largely self-contained).
\end{abstract}

\maketitle

\tableofcontents
\section{Introduction}

Let $n$ be a positive integer and set $\xn\seteq\{0,1,\ldots,n-1\}$.  We use $\xn$ to represent our standard alphabet of size $n$ and we will denote by $\shift{n}$ the usual shift map on $\xnn$. The group $\asnn$ of homeomorphisms of $\xnn$ which commute with the shift map is called \emph{the group of automorphisms of the shift dynamical system}. This is a well-studied group in symbolic dynamics, with the special property (first given by Hedlund in \cite{Hedlund69}) that if $\phi\in \asnn$ has $(x_0x_1x_2\ldots)\phi  = y_0y_1y_2\ldots$  then there is an integer $k$ so that for all indices $i$, the value $y_i$ is determined by the finite word $x_ix_{i+1}\ldots x_{i+k}$.

The paper \cite{BoyleFranksKitchens} characterises all of the finite subgroups of the group $\asnn$, shows that this group contains  non-abelian free groups whenever $n \ge 3$, and investigates other algebraic structures of the group. The papers \cite{BoyleKrieger,BoyleFiebig} develop a conjugacy invariant for the group $\asnn$, arising from the action of the group on periodic words. This invariant consists of a tuple: the well-known \emph{gyration} and \emph{sign} functions, together with \emph{first return} data: bundled data associated to the permutation representation on prime words of length $k$.  

This article resolves the following open  problem, implicit in \cite{BoyleFranksKitchens}, which Mike Boyle suggested to us for its own sake, and, as a test of our approach.

Let $\Sigma_n$ represent the group of permutations of the set $X_n$. 
 By a mild abuse of language, we say $\phi\in \asnn$ is a \emph{permutation} if there is a  fixed permutation $\alpha\in\Sigma_n$ so that if $(x_0x_1x_2\ldots)\phi=y_0y_1y_2\ldots$ then we have $y_i=(x_i)\alpha$ for all $i$. We say a permutation is a \emph{rotation} if the permutation from $\Sigma_n$ is an $n$-cycle.   We can now state the problem:  
 \begin{quote} Does there exist an automorphism $\psi\in \asnn$ of order $n$ so that all points of $\xnn$ travel on orbits of size $n$, where $\psi$ is not conjugate to a rotation?\end{quote}  
 In  \cite{BoyleFranksKitchens} Boyle, Franks and Kitchens show that if $n$ is prime then any such $\psi$ is in fact conjugate to a rotation.  We show that the Boyle, Franks, and Kitchens result holds for general $n$.
 
We have written this article so that it is essentially self-contained for general researchers working with automorphisms of the shift. 
 In particular, we gather definitions and key constructions from  \cite{OlukoyaOrder} and \cite{BleakCameronOlukoyaI} here to simplify the presentation without insisting the reader peruse those articles to follow our discussion.  We use the highlighted technology to enhance the key method in the article \cite{BoyleFranksKitchens}.  The paper \cite{BleakCameronOlukoyaI} shows how to represent any automorphism $\phi$ of the one-sided shift by a particularly nice family of transducers (finite state machines that transform inputs sequentially) while \cite{OlukoyaOrder} investigates the order problem for that same family of transducers. A key idea of \cite{BleakCameronOlukoyaI} is that any such transducer $T$ representing $\phi$ can be thought of as a triple $(D,R,\phi_*)$, where $D$ and $R$ are \emph{strongly synchronizing automata} (with $D$ representing the domain and $R$ representing the range) and where $\phi_*$ is an  isomorphism between the underlying digraphs $\Gamma(D)$ and $\Gamma(R)$ of $D$ and $R$ determined by the action of $\phi$ on periodic words.  In the case of a finite order element, the domain and range automata can also be chosen to be identical.  See Section~\ref{sec:preliminaries} for details.

In the article \cite{BoyleFranksKitchens} the central method for studying finite subgroups of $\aut{\xnn, \shift{n}}$ is firstly to find an action of the group on the underlying digraph of an automaton (now understood to be a strongly synchronizing automaton).  Once the first step is accomplished, the group is decomposed as a composition series where each composition factor is isomorphic to a subgroup of the symmetric group $\Sigma_n$ on $n$-points.  This is accomplished by pushing the action down along what is called an ``amalgamation sequence'' (see Section~\ref{Sec:discandamalg} here) of the digraph until one has an action by automorphisms on a particularly nice digraph.
The construction typically requires passing through the automorphism groups of various one-sided shifts of finite type via topological conjugations induced by the amalgamations. 

Our first step simplifies this process. In particular we show that we can always find an action of a finite subgroup of $\aut{\xnn, \shift{n}}$ on the underlying digraph of a strongly synchronizing automaton whose amalgamation and synchronizing sequences cohere (Section~\ref{Sec:discandamalg}), thus we can push down along the synchronizing sequence of that automaton without needing to possibly change alphabet. This is already enough, when $n$ is prime, to show that every element of order $p$ in $\aut{\xnn, \shift{n}}$ is conjugate in $\aut{\xnn, \shift{n}}$ to a rotation.

However, to answer the open problem above, we need to go beyond this. Suppose $\phi\in \asnn$ has order $n$ and with the condition ($\star$) that all points of $\xnn$ travel on orbits of size $n$. It turns out that ($\star$) is equivalent to the condition that for any transducer $(A,A,\phi_*)$ representing $\phi$, the action of $\phi_*$ on $\Gamma(A)$ has the property that for every (based) circuit $C$ of $\Gamma(A)$ the  orbit length of $C$ under this action is $n$. (We are using based circuits here to avoid a circuit returning to itself with some non-trivial rotation as counting as completing the orbit.) When $n$ is a prime $p$, it is not hard to see that the action of $\phi_*$ on the underlying digraph is limited in orbit lengths for edges and vertices to $1$ and $p$.  When $n$ is not prime, orbit lengths of edges and vertices can be any divisor of $n$ even though all circuits have orbit length $n$. This last issue creates problems when trying to implement the approach successfully carried out by Boyle et al for $n$ prime.

We overcome this issue for such a $\phi$ with representative transducer $(A,A,\phi_*)$ in a two step process. First, we show that  the automaton $A$ can be ``fluffed up'' by adding \textit{shadow states} (Section~\ref{Sec:Shadowstates}) to create a new strongly synchronizing automaton $B$ with an induced and more informative action $\psi_*$ on $\Gamma(B)$ so that $(B,B,\psi_*)$ still represents $\phi$. By `more informative' we mean that 
the correct addition of shadow states results in states and edges originally on orbits of length $<n$ having resulting orbits of length  $n$. Second, we take advantage of this structure to find an action by a conjugate of $\phi$ on a strongly synchronizing automaton of strictly smaller size than $A$ (where the conjugacy occurs entirely with $\asnn$).

Our approach can now be summarised as follows. First we conjugate to get a (conjugate) action of $\phi$ on a strongly synchronizing automaton whose synchronizing sequence coheres with the amalgamation sequence of its underlying digraph. Then we have a series of ``fluffing up'' moves followed by reductions via conjugation. Eventually, these processes result in a conjugate action given by a transducer over a single state automaton with $n$ labelled loops, where each edge is on an orbit of length $n$; our original element $\phi$ must then be conjugate to a rotation.  

We note that our approach is \textit{constructive}. In particular, at each step of the process described above, one can build the conjugator that enables a reduction to take place.

The property of being a strongly synchronizing automaton is equivalent to that of being a \emph{folded de Bruijn graph}.
Crucial to the approach we have sketched out is the process: given a finite order element $\phi\in \asnn$, find the minimal {folded de Bruijn graph $\Gamma$} so that $\phi$ acts faithfully on $\Gamma$ by automorphisms.  The following is essentially a result from \cite{BleakCameronOlukoyaI} stated in our context (see Lemma \ref{lem:duAA_Strongly_Synchronizing} and Theorem \ref{thm:representations}, below).

\begin{theorem}
Let $n\geq 2$ be an integer and suppose $\phi\in \asnn$ is  a finite order element.  There is an effective process for determining $\Gamma_\phi$, the minimal folded de Bruijn graph on an $n$ letter alphabet, so that $\phi$ induces a natural automorphism of $\Gamma_\phi$.
\end{theorem}

Finally, we can state the theorem which answers the question of Boyle. We actually prove a more general result.

\begin{theorem}\label{thm:mainresultIntro}
     Let $\phi \in \asnn$ be an element of finite order and let $N$ be a divisor of $n$. The following are equivalent:
     \begin{itemize}
             \item $\phi$ is conjugate to a permutation which is a product of $n/N$ disjoint cycles of length $N$;
         \item every element of $\xnn$ is on an orbit of length $N$ under the action of $\phi$; and
         \item for any folded de Bruijn graph $\Gamma_{\phi}$ admitting a faithful action by $\phi$ via an automorphism $\phi_*$, every (based) circuit of $\Gamma_{\phi}$ is on an orbit of length $N$ under $\phi_*$.
         \end{itemize}
\end{theorem}

It is unclear at the moment how much our approach depends on the condition that ``all circuits are on orbits of length $N$''. In work in progress we aim to extend our current ideas towards resolving the conjugacy problem for finite order elements of $\asnn$.

\subsection*{Acknowledgements}
The authors are grateful for partial support from EPSRC research grants EP/R032866/1. The second author is additionally grateful for support from Leverhulme Trust Research Project Grant RPG-2017-159, LMS ECF grant ECF-1920-105 and EPSRC research grant EP/X02606X/1.  Finally, we are also grateful to Mike Boyle for conversations on the question we address here.

\section{Preliminaries}\label{sec:preliminaries}
\subsection{The natural numbers and some  of its subsets}

We use the notation $\N$ for the set $\{0,1,2,\ldots \}$; for $j \in \N$ we write $\N_{j}$ for the set $\{i \in \N: i \ge j\}$ of all natural numbers which are bigger than or equal to $j$.

\subsection{Words and infinite sequences}
In this subsection we set up necessary notation for words and sequences. 

Fix a finite alphabet set $X$.

For a natural number $m$, $X^m$ is the set of ordered $m$-tuples with coordinates from $X$.  We call these the \emph{words of length $i$ (over alphabet $X$)}.  By convention, we set  $X^0\seteq\{\varepsilon\}$ and we refer to $\varepsilon$ as \emph{the empty word} or \emph{empty string}, proclaiming this to be the same object, independent of the (non-empty) set $X$ of used as our alphabet. For $u\in X^*$ we set $|u|=m$ if $u\in X^m$, noting that  $|u|$ is the length of $u$. If $u\in X^m$ then we implicitly set values $u_i\in X$ for $0\leq i<m$ so that $u=(u_0,u_1,\ldots,u_{m-1})$.  In this context, from here forward, we will simply write $u=u_0u_1\ldots u_{m-1}$. For $u \in X^{m}$ and  $i \le |u|$, we write $u_{[0,i]}$ for the prefix $u_{0} \ldots u_{i}$ of $u$. 
 
We set $X^*\seteq \cup_{n\in \N}X^n$, the words of finite length  over $X$ (this is the Kleene-star operator);  $X^+\seteq X^*\backslash\{\varepsilon\}$, the non-trivial/non-empty finite length words over $X$; and, for $i \in \N$, $X^{\ge i} \seteq \bigcup_{j \ge i} X^{i}$, the set of words of length at least $i$. Analogously,  $X^{\le i}$ is the set of words of length at most $i$.

For $u,v\in X^*$, so that $u=u_0u_1\ldots u_{l-1}$ and $v=v_0v_1\ldots v_{m-1}$, $uv$ will represent the concatenation of these words: $uv\seteq u_0u_1\ldots u_{l-1}v_0v_1\ldots v_{m-1}$, which is a word of length $l+m$ over $X$.

Almost always, we will take $X_n := \{0,1,\ldots, n-1\}$ as our alphabet set of size $n$. In such a scenario, we often equip $X_n$ with the natural linear order $0 < 1 < \ldots < n-1$ and, $X_n^*$  with the induced dictionary order.  

\subsection{One-sided shift space}
Following the paper \cite{BleakCameronOlukoyaI}, we take $\xnN\seteq\{\ldots x_{-2}x_{-1}x_0\mid x_i\in \xn\}$ 
as our shift space, with the shift operator $\shift{n}$ defined by $(x_i)_{i\in-\N}\shift{n}=(y_i)_{i\in -\N}$
where we have $y_i= x_{i-1}$.   Thus, 
for a finite-length word over $\xn$ we may index this word with negative or positive indices as seems natural at the time.  When we are explicitly thinking of a finite subword $w\in \xn^k$ of a point $x\in\xnN$ we will ordinarily index $w$ as $w=w_{i-k+1}w_{i-k+2}\ldots w_{i}$ for some $i\in -\N$. 

Suppose $k$ is a positive integer and $u=u_{-(k-1)}u_{-(k-2)}\ldots u_{-1}u_0\in \xn^k$. We write $u^\omega$ for the element $x:=\ldots x_{-m}x_{-m+1}\ldots x_{-1}x_0  \in \xnN$, uniquely defined by the following rule: for all $i \in -\N$, $x_i=u_{-a}$ for $a$ between $0$ and $k-1$ congruent to $i$ modulo $k$. The element $x \in \xnN$ is called a \textit{periodic element}. The \textit{period} of $x$ is the smallest $j \in \N_1$ such that $(x)\shift{n}^{j} = x$.

\subsection{Automata and transducers}\label{subsec:autTran}

\subsubsection{Automata}
An \emph{automaton}, in our context, is a triple $A=(X,Q_A,\pi_A)$, where
\begin{enumerate}
\item $X$ is a finite set called the \emph{alphabet} of $A$;
\item $Q_A$ is a finite set called the \emph{set of states} of $A$;
\item $\pi_A$ is a function $X\times Q_A\to Q_A$, called the \emph{transition
function}.
\end{enumerate}

The \emph{size} of an automaton $A$ is the cardinality of its state set.  We use the notation $|A|$ for the size of the $A$.

We regard an automaton $A$ as operating as follows. If it is in state $q$ and
reads symbol $x$ (which we suppose to be written on an input tape), it moves
into state $\pi_A(x,q)$ before reading the next symbol. As this suggests, we
can imagine that the automaton $A$ is in the middle of an input word, reads the next letter and  moves to the right, possibly changing state in the process.

Mainly we consider automaton over an alphabet $\xn$ although there are some instances where $\xn$ is not the alphabet.

We can extend the domain of the transition function as follows. For $w\in X^m$, let $\pi_A(w,q)$ be
the final state of the automaton which reads the word $w$ from initial state $q$. Thus, if $w=w_0w_1\ldots w_{m-1}$, then
\[\pi_A(w,q)=\pi_A(w_{m-1},\pi_A(w_{m-2},\ldots,\pi_A(w_0,q))\ldots)).\]
By convention, we take $\pi_A(\varepsilon,q)=q$.

An automaton $A$ can be represented by a labelled directed graph $\Gamma_A$, whose
vertex set $V_A$ is $Q_A$.  For this directed graph there is a directed edge labelled by $x\in X$ from
$p$ to $q$ if $\pi_A(x,p)=q$.  Representing this, we determine the set $E_A$ of edges of $\Gamma_A$ to be the set of triples
\[ E_A \seteq \{(p,x,q)\mid \exists p,q\in Q_A, x\in X,\textrm{ so that } \pi_A(x,p)=q\}.
\] 
We have a labelling map $\mathfrak{l}_A: E_A \to X$ by $\mathfrak{l}_A((p,x,q) = x$. Thus $\Gamma_A = (V_A, E_A, \mathfrak{l}_{A})$. Note that for each $p, \in Q_A$ and $s \in X$, there is one and only one edge $(p,s,q)$.

Write $G_A$ for the unlabelled digraph $(V_A, E_A)$. The graph $\Gamma_A$ will be referred to as the \emph{underlying labelled digraph for the automaton $A$}; the graph $G_A$ will be referred to as the \textit{underlying digraph of $A$.}

\subsubsection{Transducers}
For our purposes, a \emph{transducer} is a quadruple $T=(X,Q_T,\pi_T,\lambda_T)$, where
\begin{enumerate}
\item $(X,Q_T,\pi_T)$ is an automaton;
\item $\lambda_T:X \times Q_T\to X$ is the \emph{output function}.
\end{enumerate}
Formally, such a transducer is an automaton which can write as well as read. After
reading symbol $a$ in state $q$, it writes a new symbol $\lambda_T(a,q)$ on an
output tape, and makes a transition into state $\pi_T(a,q)$. 

The size of a transducer is the size of its underlying automaton.

As with automata, we extend the domain of the transition function $\pi_{T}$ to $X^{\ast}\times Q_T$. We similarly extend the domain of the output function $\lambda_{T}$: 
for an element  $ w \in X^{\ast} \sqcup X^{\N}$ and a state $q \in Q_T$ we inductively define $\lambda_{T}(w,q)$ by the rule:
$$\lambda_T(w,q):=\lambda_{T}(w_0, q)\lambda_{T}(w-w_0, \pi_{T}(w_0,q))$$ 
where $w_0 \in X$ and $w-w_0 \in X^{\ast} \sqcup X^{\N}$ are uniquely defined by the equality $w_0(w-w_0) = w$. 
By convention, we take $\lambda_{A}(\varepsilon, q) = \varepsilon$. 
Thus, for each state $q \in Q_T$, and each $k \in \N$, $\lambda_{T}(\cdot,q)$ is a map from $X^{k}$ to itself as well as a continuous map from the Cantor space $X^{\N}$ to itself (equipping $X$ with the discrete topology and taking the product topology on $X^{\N}$).

For ease of notation we will normally denote the map $\lambda_{T}(\cdot, q)$ by $T_{q}$. We will dually think of this map $T_{q}$ as the transducer $T$ with the distinguished initial state $q$.

A transducer $T$ can also be represented as an edge-labelled directed graph.
Again the vertex set is $Q_T$; now, if $\pi_T(x,q)=p$, we put an edge with
label $x|\lambda_T(x,q)$ from $q$ to $p$. In other words, the edge label describes both the input and the output associated with that edge. We call $x$ the \textit{input label} of the edge and $\lambda_{T}(x,q)$ the \textit{output label} of the edge.

For example, Figure~\ref{fig:shift2} describes a synchronous transducer over the alphabet
$X_2$.

\begin{figure}[htbp]
\begin{center}
\scalebox{0.85}{
 \begin{tikzpicture}[shorten >=0.5pt,node distance=3cm,on grid,auto]
 \tikzstyle{every state}=[fill=none,draw=black,text=black]
    \node[state] (q_0)   {$a_1$};
    \node[state] (q_1) [right=of q_0] {$a_2$};
     \path[->]
     (q_0) edge [loop left] node [swap] {$0|0$} ()
           edge [bend left]  node  {$1|0$} (q_1)
     (q_1) edge [loop right]  node [swap]  {$1|1$} ()
           edge [bend left]  node {$0|1$} (q_0);
 \end{tikzpicture}
 }
 \end{center}
 \caption{A transducer over $X_2$ \label{fig:shift2}}
\end{figure}

In what follows, we use the language \emph{automaton} only for those automata which are not transducers.

\subsubsection{Transducer toolkit}\label{sec:transducertoolkit}

Let $T$ and $U$ be (not necessarily distinct) transducers over the alphabet $X$ and $(p,q) \in Q_T \times Q_u$. The states $p$ and $q$ are said to be \emph{$\omega$-equivalent} if the
transducers $T_{p}$ and $U_{q}$ induce the same continuous map $X^{\N}$. (This can
be checked in finite time, see~\cite{GNSenglish}.)  In the case that $p$ and $q$ are $\omega$-equivalent states, we say that the two transducers $T_p$ and $U_{q}$ are \emph{$\omega$-equivalent}. Note that in this case, for any $x \in X^{\ast}$, $\lambda_T(x, p) = \lambda_U(x,q)$ and the states $\pi_{T}(x, p)$ and $\pi_{U}(x, q)$ are again $\omega$-equivalent states.

A transducer is said to be \emph{minimal} if no two of its states are
$\omega$-equivalent. 

 Two minimal non-initial transducers $T$ and $U$  are said to be  \emph{$\omega$-equal} if there is a bijection $f: Q_{T} \to Q_{U}$, such that for any $q \in Q_{T}$, $T_{q}$ is $\omega$-equivalent to $U_{(q)f}$. Notice that such a bijection must be unique. 
 
 Two  minimal initial transducers $T_{p}$ and $U_{q}$ are said to be $\omega$-equal if they are $\omega$-equal as non-initial transducers and for the bijection $f: Q_{T} \to Q_{U}$ witnessing this, $(p)f = q$. We use the symbol `$=$' to represent $\omega$-equality of initial and non-initial transducers. Two non-initial transducers $T$ and $U$ are said to be \emph{$\omega$-equivalent} if they have $\omega$-equal minimal representatives

In the class of synchronous transducers, the  $\omega$-equivalence class of  any transducer has a unique minimal representative.

Throughout this article, as a matter of convenience, we shall not distinguish between $\omega$-equivalent transducers. Thus, for example, we introduce various groups as though their group elements are transducers, whereas the elements of these groups are in fact  $\omega$-equivalence classes of transducers. 

Given two transducers $T=(X,Q_T,\pi_T,\lambda_T)$ and
$U=(X,Q_U,\pi_U,\lambda_U)$ with the same alphabet $X$, we define their
product $T*U$. The intuition is that the output for $T$ will become the input
for $U$. Thus we take the alphabet of $T*U$ to be $X$, the set of states
to be $Q_{T*U}=Q_T\times Q_U$, and define the transition and rewrite functions
by the rules
\begin{eqnarray*}
\pi_{T*U}(x,(p,q)) &=& (\pi_T(x,p),\pi_U(\lambda_T(x,p),q)),\\
\lambda_{T*U}(x,(p,q)) &=& \lambda_U(\lambda_T(x,p),q),
\end{eqnarray*}
for $x\in X$, $p\in Q_T$ and $q\in Q_U$.

In automata theory, a synchronous (not necessarily initial) transducer $T = (X, Q_{T}, \pi_{T}, \lambda_T)$ is \emph{invertible} if for any state $q$ of $T$, the map $\rho_q:=\lambda_{T}(\centerdot, q): X \to X$ is a bijection. In this case the inverse of $T$ is the transducer $T^{-1}$ with state set $Q_{T^{-1}}:= \{ q^{-1} \mid q \in Q_{T}\}$, transition function $\pi_{T^{-1}}: X \times Q_{T^{-1}} \to Q_{T^{-1}}$ defined by $(x,p^{-1}) \mapsto q^{-1}$ if and only if $\pi_{T}((x)\rho_{p}^{-1}, p) =q$, and output function  $\lambda_{T^{-1}}: X \times Q_{T^{-1}} \to X$ defined by  $(x,p) \mapsto (x)\rho_{p}^{-1}$. Thus, in the graph of the transducer $T$ we  simply switch the input labels with the output labels and append `${}^{-1}$' to the state names.

Since each state of an invertible transducer induces a permutation of $X$, one can think of an invertible transducer $T$ as `gluing'' together a pair of automata along a digraph isomorphism. Said differently, an invertible transducer can be split off as a \textit{domain automaton} $(X, Q_T, \pi_T)$ and a \textit{range automaton} $(X, Q_{T^{-1}},\pi_{T^{-1}})$ with the ``gluing'' isomorphism  induced by the mapping $^{-1}: Q_T \to Q_{T^{-1}}$. This is a point of view we exploit in this work. Typically, in this approach, we identify $Q_T$ with $Q_{T^{-1}}$, that is, we will not distinguish between a state $q \in Q_T$ and the corresponding state $q^{-1}$ in $Q_{T^{-1}}$. This makes sense when one considers the graph of an invertible transducer: the graph of the domain automaton is obtained by deleting the output on transitions (so an edge label $x|y$ becomes the label $x$); dually the graph of the range automaton is obtained by deleting inputs on transitions (an edge label $x|y$ becomes the label $y$.

We extend the usual index laws to transducers with the product defined above. For example, we write $T^{m}$ for the $m$-fold product $T \ast T \ast \ldots \ast T$.

We are concerned only with {\textbf{invertible, synchronous transducers}} in this article.

\subsection{Automorphisms of the one-sided shift and synchronizing automata }

We briefly outline how one can represent an automorphism of the dynamical system $(\xnN, \shift{n})$ using a transducer over the alphabet $\xn$ which satisfies a strong synchronization condition. We note briefly, that synchronization occurs in automata theory
in considerations around the \emph{\v{C}ern\'y conjecture}. In that context, a word $w$ is said to be a \emph{reset word} for $A$ if $\pi_A(w,q)$ is
independent of $q$; an automaton is called \emph{synchronizing} if it has
a reset word~\cite{Volkov2008,ACS}. The strong synchronization condition alluded to, requires that every word of length $k$ to be a reset word for the automaton. For a fuller exposition see \cite{BleakCameronOlukoyaI}.

\subsubsection{Synchronizing automata}
Given a natural number $k$, we say that an automaton $A$ with alphabet $X_n$
is \emph{synchronizing at level $k$} if there is a map
$\mathfrak{s}_{k}:X_n^k\mapsto Q_A$ such that for any $q \in Q_A$ and any 
$w\in X_n^k$, we have $\pi_A(w,q)=\mathfrak{s}_{k}(w)=: q_w$. In other words, $A$ is synchronizing at level $k$ if, after reading a word $w$ of length $k$ from a state $q$, the final state depends only on
$w$.  We
call $q_w$ the state of $A$ \emph{forced} by $w$ and $\mathfrak{s}_{k}$ the synchronizing map (of $A$) at level $k$. An automaton $A$ is called \emph{synchronizing} if it is synchronizing at level $k$ for some $k$.

Clearly if an automaton is synchronizing at level $k$, then it is synchronizing at level $l$ for all $l \ge k$. Thus, for a synchronizing automaton $A$, we may define a map $\mathfrak{s}_{A}: \xn^{\ge k} \to Q_A$ by $\mathfrak{s}_{A}(w) = \mathfrak{s}_{|w|}(w)$. Notice that the image 
$\im{\mathfrak{s}_{A}}$ of $\mathfrak{s}_{A}$ is an 
inescapable set of states of $A$: for any $q \in \im{\mathfrak{s}_{A}}$ and any $x \in \xn$, $\pi_{A}(x,q) \in \im{\mathfrak{s}_{A}}$.  The following definition consequently makes sense.

Let $A$ be a synchronizing automaton, define the
\emph{core} of $A$, denoted $\core(A)$, to be the sub-automaton $(\xn, \im{\mathfrak{s}_{A}}, \pi_A)$ of $A$ with state set equal to the image of the map
$\mathfrak{s}_{A}$. This is a well-defined synchronizing automaton that depends only on $A$. 

We say that a synchronizing automaton or transducer is \emph{core} if it is equal to its core. 

Let $T$ be a transducer which is invertible with inverse $T^{-1}$. If $T$ is synchronizing at level $k$, and $T^{-1}$ is synchronizing at level $l$, 
we say that $T$ is \emph{bisynchronizing} at level $(k,l)$; if $T$ is bi-synchronizing at level $(k,k)$ for some $k$, then simply say $T$ is bi-synchronizing at level $k$.

We note that the product of two synchronizing transducers is again synchronizing (\cite{AutGnr}).

\subsubsection{Automorphisms of the one-sided shift from synchronizing transducers}

Let $T$ be a transducer which (regarded as an
automaton) is synchronizing at level $k$, then the core of $T$ (similarly denoted 
 $\core(T)$) induces a continuous map $$f_T:\xnN\to\xnN$$ as follows. Let $x \in \xnN$ and set $y \in \xnN$ to be the sequence defined by $$y_{i} = \lambda_T(x_i, q_{x_{i-k}x_{i-(k-1)}\ldots x_{i-1}}).$$ Note that $$\pi_{T}(x_{i}, q_{x_{i-k}x_{i-(k-1)}\ldots x_{i-1}}) = q_{x_{i-(k-1)}\ldots x_{i-1}x_i}.$$ Set $$(x)f_{T} = y.$$ Thus, from the point of view of the transition function of $T$ we in fact begin processing $x$ at $-\infty$ and move towards $x_0$.  (This is in keeping with our interpretation of transducer as representing machines applying sliding block codes, where here, we are thinking of $\asnN$ as consisting of the sliding block code transformations that require past information only to determine what to do with a digit.) Note, moreover, that the map $f_T$ is independent of the (valid) synchronizing level chosen to define it. 
We have the following result:

\begin{prop}\cite{BleakCameronOlukoyaI}\label{prop:pntildeisinendo}
Let $T$ be a minimal transducer which is  synchronizing at 
level $k$ and which is core. Then $f_T\in\End(\xnN,\shift{n})$.
\end{prop}

The transducer in Figure~\ref{fig:shift2} induces the shift map on $\xnN$. 

In \cite{AutGnr}, the authors show that the set $\shn{n}$ of
minimal finite synchronizing invertible synchronous core transducers is a monoid; the monoid operation
consists of taking the product of transducers and reducing it by removing non-core states and identifying $\omega$-equivalent states to obtain a minimal and synchronous representative. 

Let $\mathcal{H}_n$ be the subset
of $\shn{n}$ consisting of transducers which are bi-synchronizing.  A result of \cite{BleakCameronOlukoyaI} is that $\asnN \cong \mathcal{H}_n$. 

\subsection{Finite order elements of \texorpdfstring{$\hn{n}$}{Lg} as automorphisms of folded de Bruijn graphs}\label{Sec:debruijnandfolded}

In this subsection we introduce the main way we view finite order elements  of $\hn{n}$. The results in this section are based on the article \cite{BleakCameronOlukoyaI}.

\subsubsection{de Bruijn graphs, foldings and synchronizing automata}
For integers
$m\ge1$ and $n\ge2$, the \emph{de Bruijn graph} $G(n,m)$ can be defined as follows. The vertex set is $X_n^m$, for vertices $v,w \in \xn^{m}$, there is a directed edge from $v$ to $w$ if and only if
$v_1v_2\ldots v_{m-1} = w_0w_1\ldots w_{m-2}$. In this case, the edge from $v$ to $w$ has label $w_{m-1}$.

Figure~\ref{fig-DB-3-2-straight} shows the de Bruijn graph $G(3,2)$.
\begin{figure}[htb]
\begin{center}
\scalebox{0.75}{
\begin{tikzpicture}[->,>=stealth',shorten >=1pt,auto,node distance=2.3cm,on grid,semithick,every state/.style={draw=black,text=black},scale=1]

   \node[at={(0,2.9)},state] (a) {$00$}; 
   \node[at={(-3.3,-3.3)},state] (b)  {$11$}; 
   \node[at={(3.3,-3.3)},state] (c) {$22$}; 
   \node[at={(-1.2,-0.5)},state] (d)   {$01$}; 
   \node[at={(-3.0,0.6)},state] (f)  {$10$}; 
   \node[at={(3.0,0.6)},state] (e) {$02$}; 
   \node[at={(1.2,-0.5)},state] (h)  {$20$}; 
   \node[at={(0,-2.8)},state] (g)  {$12$}; 
   \node[at={(0,-4.7)},state] (i){$21$}; 

    \path (a) edge [out=70,in=110,loop,min distance=0.5cm]node [swap]{$0$} (a);
    \path (a) edge node [swap]{$1$} (d);
    \path (a) edge node {$2$} (e);
    \path (b) edge [out=205,in=245,loop,min distance=0.5cm]node [swap]{$1$} (b);
    \path (b) edge node {$0$} (f);
    \path (b) edge node [swap] {$2$} (g);
    \path (c) edge [out=-65,in=-25,loop,min distance=0.5cm]node [swap]{$2$} (c);
    \path (c) edge node {$1$} (i);
    \path (c) edge node [swap]{$0$} (h);
    \path (d) edge node [swap] {$1$} (b);
    \path (d) edge [bend right=15] node [swap]{$0$} (f);
    \path (d) edge node[swap] {$2$} (g);
    \path (e) edge [bend left=100,min distance=4.45cm] node {$1$} (i);
    \path (e) edge [bend right=15] node [swap] {$0$} (h);
    \path (e) edge  node {$2$} (c);
    \path (f) edge node {$0$} (a);
    \path (f) edge [bend right=15] node [swap]{$1$} (d);
    \path (f) edge [bend left=100,min distance=4.45cm] node {$2$} (e);
    \path (g) edge node[swap] {$0$} (h);
    \path (g) edge node [swap] {$2$} (c);
    \path (g) edge [bend right=15] node [swap] {$1$} (i);
    \path (h) edge [bend right=15] node [swap] {$2$} (e);
    \path (h) edge node [swap] {$0$} (a);
    \path (h) edge node [swap] {$1$} (d);
    \path (i) edge [bend left=100,min distance=4.45cm] node {$0$} (f);
    \path (i) edge [bend right=15] node [swap] {$2$} (g);
    \path (i) edge node {$1$} (b);
\end{tikzpicture}
}
\end{center}
\caption{The de Bruijn graph $G(3,2)$.\label{fig-DB-3-2-straight}}
\end{figure}

Observe that the de Bruijn graph $G(n,m)$ describes an automaton over the
alphabet $X_n$. Moreover, this automaton is synchronizing at level $m$: for $w \in \xn^{m}$, $q_{w} = w$.

The de Bruijn graph is the universal automaton (with respect to \textit{foldings})
over $X_n$ which is synchronizing at level $m$.

A \emph{folding} of an automaton $A$ over the alphabet $X_n$ is
an equivalence relation $\equiv$ on the state set of $A$ with the property
that reading the same letter from equivalent states takes the automaton
to equivalent states. More precisely, if $v\equiv v'$ and $\pi_A(x,v)=w$, $\pi_A(x,v')=w'$, then $w\equiv w'$.
If $\equiv$ is a folding of $A$, then we can uniquely
define the \emph{folded automaton} $A/{\equiv}$: the state set is the set
of $\equiv$-classes of states of $A$, and, denoting the $\equiv$-class of
$w$ by $[w]$, we have $\pi_{A/{\equiv}}(x,[w])=[\pi_A(x,w)]$.

\begin{prop}\cite{BleakCameronOlukoyaI}
The following are equivalent for an automaton $A$ on the alphabet $X_n$:
\begin{itemize}\itemsep0pt
\item $A$ is synchronizing at level $m$, and is core;
\item $A$ is the folded automaton from a folding of the de Bruijn graph $G(n,m)$.
\label{p:fold}
\end{itemize}
\end{prop}

We may think of a de Bruijn graph $G(n,m)$ as determining a finite category, with objects the foldings of $G(n,m)$ and with arrows digraph morphisms which commute with the transition maps of the given automata.  It is immediate in that point of view that all such arrows are surjective digraph morphisms (and indeed, these are folding maps).  

\subsubsection{Finite order elements of \texorpdfstring{$\hn{n}$}{Lg} as automorphisms of synchronizing automata}\label{sec:FiniteSubgroups}\label{sec:finiteByGraphAuto}

Let $A$ be a finite automaton on edge-alphabet $X_n$. Let $E_{A} \subset    Q_{A} \times X_n\ \times Q_{A}$ be the set of edges of $\Gamma_{A}$, and $G_{A}$ be the underlying unlabelled digraph of $A$. Recall, that an edge $(p, x, q) \in E_{A}$ has the label `$x$' in $\Gamma_{A}$.

Let $\phi$ be an automorphism of the directed graph $G_{A}$. The action of $\phi$ on $E_A$ uniquely determines a function, 
$\lambda_{\mathscr{H}(A,\phi)}: \xn \times Q_{A} \to \xn$ by the rule $\lambda_{\mathscr{H}(A,\phi)}(x, p)= y$ if and only if for the unique $q \in Q_A$ such that $(p,x,q) \in E_A$, $(p,x,q)\phi = (r,y,s)$. 

The transducer $\mathscr{H}(A, \phi)$, which  is invertible can be thought of as the result of gluing the automaton $A$ to a copy of itself along the map $\phi$ (see the discussion the definition of an invertible transducer in Section~\ref{sec:transducertoolkit}). In particular, if $(p,s,q) \in E_{A}$, and $(p,x,q)\phi = (p\phi,y\phi,q\phi)$, then 
\begin{itemize}
    \item the vertices $p$ and $(p)\phi$ are identified;
    \item the vertices $q$ and $(q)\phi$ are identified;
     \item the edges $(p,x,q)$ and are identified to a single edge $((p,x,q))\phi$ labelled $x | y$. 
\end{itemize}

Note that the domain and range automaton of $\mathscr{H}(A,\phi)$ are both equal to $A$.

\begin{rem}\label{rem:representationProps}
Suppose $A$ is synchronizing at level $k$ and core. We make a few observations. 
\begin{enumerate}
	
	\item  Both $\mathscr{H}(A, \phi)$ and $\mathscr{H}(A, \phi)^{-1} = \mathscr{H}(A, \phi^{-1})$ are synchronizing at level $k$ hence the minimal  representative $\overline{\mathscr{H}(A, \phi)}$ of $\mathscr{H}(A, \phi)$ is an  element of $\hn{n}$. 
	
	\item In fact, for a state $q \in Q_{A}$, if $$W_{k,q}:= \{ w \in X_{n}^{k}: \pi_{\mathscr{H}(A, \phi)}(w,q) = q \}$$ is the set of words of length $k$ that force the state $q$,  then $$\{\lambda_{\mathscr{H}(A, \phi)}(w, p) \mid w \in W_{k,q}, p \in Q_{\mathscr{H}(A, \phi)} \}= W_{k, (q)\phi}$$ is the set of words of length $k$ that force the state $(q)\phi$ of $A$,
	
		\item \label{rem:repProp-finOrder} An element of $\hn{n}$ which can be represented by a transducer $\mathscr{H}(A,\phi)$ for some folded de Bruijn graph $A$ and digraph automorphism $\phi$ of $G_A$ must have finite order. 
\end{enumerate}
\end{rem}

Let $H\in \hn{n}$. Suppose $A$ is a synchronizing and core automaton admitting a digraph automorphism $\phi: G_A\to G_A$ so that $H$ and $\mathscr{H}(A,\phi)$ represent the same transformation of $\xnN$ then we say \emph{$H$ is induced from $(A,\phi)$}. We further  say that the synchronizing automaton \emph{$A$ is an automaton supporting/is a support of $H$ (realised by $\phi$)}.

\section{Minimal digraphs for finite order elements of \texorpdfstring{$\hn{n}$}{Lg}} 

In this  section, we demonstrate constructively that for a finite order element $H \in \hn{n}$, there is a unique smallest synchronizing automaton $A$ supporting $H$. 

The proof of this result really is a strengthening of Theorem 4.5 of \cite{BleakCameronOlukoyaI} and, as in that article, we require some facts about duals of synchronizing transducers. For more on duals of general transducers  one can consult \cite{Klimann2012,AkhaviKlimannLombardyMairessePicantin2012,NekrashevychSSG}; for more on duals of synchronizing transducers one can consult \cite{OlukoyaOrder}.

\subsection{The dual transducer}\label{subsec:alphaAndDual}

Fix $T = (\xn, Q_T, \lambda_{T}, \pi_{T})$ a transducer. Define functions $\dpi_{T}: Q_{T} \times \xn \to \xn$ and $\dlambda_{T}: Q_T \times \xn \to \xn$ as follows. For $q \in Q_{T}$ and $x \in \xn$, $\dpi_{{T}}(q, x) = y$ if and only if $\lambda_{T}(x, q) = y$ and $\dlambda_{{T}}(q, x) = p$ if and only if $\pi_{T}(x,q) = p$.  The transducer $\dual{T} = \gen{Q_{T}, X_{n}, \dpi_{T}, \dlambda_{T}}$ is the \textit{dual transducer of $T$}. The state set of $\dual{T}$ is the set $X_{n}$, the alphabet of $\dual{T}$ is the state set $Q_{T}$ of $T$.

We often need to consider powers (that is powers under the transducer product defined in Section~\ref{sec:transducertoolkit}) of the dual transducer $\dual{T}$,  and the following construction, the `paths to letters' construction,  simplifies how to compute this from $T$ (see \cite{AkhaviKlimannLombardyMairessePicantin2012} where it plays a key role).

Let $m \in \N_{1}$.  Define functions $\pi_{T(m)}: \xn^{m} \times Q_{T} 
\to Q_{T}$ by $\pi_{T(m)}(w, q) = \pi_{T}(x, q)$ and 
$\lambda_{T(m)}: \xn^{m} \times Q_{T}$ by $\lambda_{T(m)}(w,q) = 
\lambda_{T}(w,q)$. The transducer $T(m) = (\xn^{m}, Q_{T}, 
\pi_{T(m)}, \lambda_{T(m)})$ is then the transducer $T$ where 
the input alphabet is words over the alphabet $m$. The $m$\textsuperscript{th} power of $\dual{T}$, is the dual of $T(m)$.

\begin{lemma}
    Let $T$ be a synchronous transducer over alphabet $X_n$.  For positive natural $m$, we have $(\dual{T})^{m}=\dual{T(m)}$.
\end{lemma} 
To lighten our notation below, we use the notation $\dual{T}_{m}$ for the transducer $\dual{T(m)}$.

Also observe that $\dual{T^{-1}}$ is obtained from $\dual{T}$ by `reversing the arrows'. That is, for $x,y \in \xn$, $q, p \in Q_{T}$  such that $\dpi_{T}(q, x) = y$ and $\dlambda(q, x) =p$, then $\dpi_{T^{-1}}(q^{-1}, y) = x$ and  $\dlambda(q^{-1}, y) =p^{-1}$.

\subsection{Finite order elements of \texorpdfstring{$\hn{n}$}{Lg}}

 We are now in a position to construct for a finite order element $H \in \hn{n}$ and corresponding $\phi_H \in \asnN$, the minimal synchronizing automaton $A$ admitting an automorphism $\phi$ such that $H$ is induced from  $(A,\phi)$.

We have some potential collisions of notation that we ought to clarify.  

Firstly $\phi\in \asnN$, then we can represent $\phi$ by a (minimal) transducer $H_\phi\in \hn{n}$. (Determining $H_\phi\in \hn{n}$ from a given element $\phi\in \asnN$ is not difficult --- one simply associates states to maps determined by local actions of the sliding block code. 
 See \cite{BleakCameronOlukoyaI} or \cite{BleakCameronOlukoyaII}  for details.)

Secondly, if $H\in \hn{n}$, then $H$ represents an element $\phi_H\in \asnN$.

Finally, if $\phi\in \asnN$ has finite order, then as we will see from Theorem \ref{thm:representations} there is a synchronizing automaton $A$ and an automorphism $\psi$ of the underlying digraph $G_{A}$ so that $H_\phi$ and $\mathscr{H}(A,\psi)$ represent the same element.  It happens that there is a way to define $\psi$ from $\phi$, and also,  from $H_\phi$.  Similarly, we could have begun this paragraph with an element $H\in \hn{n}$, in which case $\psi$ would be defined from both $\phi_H$ and from $H$.

To unify our notation, we will write $\phi_H$ for both the maps $\psi$ and $\phi_H$ as in the above situation.  This of course means that $\phi_H$ will represent two different things (an automorphism of the one-sided shift, or alternatively, an automorphism of a digraph underlying a folded de Bruijn graph).  It will be clear what is meant from the context.

For the remainder of this section, \textbf{fix a finite order element $H \in \hn{n}$}. We begin by constructing an automaton $\duA{H}$ admitting an automorphism $\phi$ such that $H$ is induced from $(\duA{H},\phi)$.  

\subsubsection{Building \texorpdfstring{$\duA{H}$}{Lg} from \texorpdfstring{$H$}{Lg}}\label{Subsection:buildingautomatonandautomorphism}

The construction that follows is illustrated in Example~\ref{Example:minrepfromdual} for an element of $\hn{6}$ of order $6$.

 Consider the semigroup $ \{\overline{\dual{H}_{i}} : i \in \N_1\}=:\gen{\overline{\dual{H}}}$ of minimal representatives (as in Subsection~\ref{sec:transducertoolkit}) of powers of $\dual{H}$. In 
\cite[Proposition 4.15]{OlukoyaOrder} it is shown that there is a $k \in \N$ 
such that $\rdual{H}$ is the zero $\gen{\overline{\dual{H}}}$, that is, $\overline{\dual{H}_{j}} = \overline{\dual{H}_{k}}$ for all $j \in \N_{k}$. Fix the minimal $k \in \N$ where this happens. By construction, for $\gamma \in \xn^{k}$, a state $[\gamma]$ of $\rdual{H}$ is the $\omega$-equivalence class of the state $\gamma$ of $\dual{H}_{k}$. Since $\overline{\dual{H}_{j}} = \overline{\dual{H}_{k}}$ for all $j \in \N_{k}$, we may in fact identify a state $[\gamma]$ of $\rdual{H}$ with the set of all words $\delta \in \xn^{\ge k}$ such that the initial transducers $(\dual{H}_{|\delta|})_{\delta}$ and $(\rdual{H})_{[\gamma]}$ are $\omega$-equivalent.

The following is a {\bf very useful fact} from \cite{OlukoyaOrder}: for every state $[\gamma]$ of the zero
$\overline{\dual{H}_{k}}$, there is a word $W([\gamma]) \in Q_{H}^{+}$ such that for any input word $s \in Q_{H}^{+}$, the output when $s$ is processed from the state $[\gamma]$ 
of $\overline{\dual{H}_{k}}$ is the word 
$(W([\gamma]))^{l} W([\gamma])_{[0,m-1]}$, where, $|s|  = l|W([\gamma])| + m$ and $W([\gamma])_{[0,m-1]}$ is the length $m$ prefix of $(W([\gamma]))$.  It follows from this that $\overline{\dual{H}_{k}}$ has the following structure:  for each state $[\gamma]$ of $\overline{\dual{H}_{k}}$  there is a $q_{[\gamma]}\in Q_H$ so that for any $p\in Q_H$ 

\begin{itemize}
    \item $\pi_{\overline{\dual{H}_{k}}}(p,[\gamma])= [\lambda_{H}(\gamma, p)] = [\gamma] H$, and
    \item $\lambda_{\overline{\dual{H}_{k}}}(p,[\gamma])= q_{\gamma} = q_{[\gamma]}$.
    \end{itemize}

Form the automaton $\duA{H}$ as follows. The states of $\duA{H}$ are the states $[\gamma]$ of $\overline{\dual{H}_{k}}$ and the transitions are given by the rule that for $x \in \xn$, and $[\gamma]$ a state of 
$\overline{\dual{H}_{k}}$, we set  $$\pi_{\duA{H}}(x, [\gamma]) =  
[\gamma_{[1, |\gamma|-1]}x].$$ By the observation above, it does not matter which $\gamma\in \xn^{\ge k}$ one picks from the class $[\gamma]$. 

The following lemma is immediate from the definitions:
\begin{lemma}\label{lem:duAA_Strongly_Synchronizing}
    Let $H\in \hn{n}$ be an element of finite order then the automaton $\duA{H}$ is synchronizing at level $k$.
\end{lemma}
We also have the following translation (into our context) of the statement of Theorem 4.5 of  \cite{BleakCameronOlukoyaI}: $A(G)$ in that theorem corresponds to $\duA{H}$ here, and  $G$ there is the group $\langle H\rangle$ here.

\begin{theorem}\label{thm:representations}
Let $H\in \hn{n}$ be an element of finite order. Then
\begin{enumerate}
	\item $H$ acts as an automorphism $\phi_{H}$ of the digraph underlying $\duA{H}$ by mapping an edge $([\gamma], x, [\gamma_{[1,|\gamma|-1]}x])$ to the edge $(([\gamma])H, \lambda_{H}(x, q_{\gamma}), ([\gamma_{[1,|\gamma|-1]}x])H )$ where $([\gamma])H = [\lambda_{H}(\gamma, q)]$ for some $q \in Q_{H}$;
	\item the minimal representative of the transducer $\mathscr{H}(\duA{H},\phi_{H}^{i})$ is the element $H^i \in \hn{n}$.
\end{enumerate}
\end{theorem}

Ahead of an illustrative example, we make the following note. A circuit of length $k$ in an automaton $A$ is carried by $k$ directed edges $(e_0,e_1,\ldots, e_{k-1})$ in $G_A$  where the target of $e_i$ is the source of $e_{i+1}$ (indices modulo $k$) for each index $i$.  In what follows, an automorphism $\phi$ of $G_A$ carries a circuit $\mathcal{C}$ to itself if and only if the image of each edge of $\mathcal{C}$ is itself under the automorphism.  Specifically, we do not consider a \textbf{``rotation'' of a circuit to be the circuit itself}.

\begin{example}\label{Example:minrepfromdual}
The transducer $H$ depicted in Figure \ref{fig:ExampleACallup} is  an element of $\hn{6}$ and bi-synchronizing at level 2.
\begin{figure}[htb]
\begin{center}
\scalebox{0.85}{
\begin{tikzpicture}[shorten >=0.5pt,node distance=3cm,on grid,auto] 
    \node[state] (q_0)   {$a_0$}; 
    \node[state] (q_1) [xshift=4cm,yshift=0cm] {$a_1$};
   \node[state] (q_3) [xshift=4cm,yshift=-3cm] {$a_2$}; 
   \node[state] (q_5)[xshift=8cm,yshift=-3cm] {$a_3$}; 
    \path[->] 
    (q_0) edge node {$1|4$ $0|5$} (q_1)
          edge [out=330,in=135]  node [xshift=-0.3cm,yshift=0.cm] {$2|0$} node  [xshift=0.3cm,yshift=-0.4cm] {$3|1$} (q_3)
          edge [out=50,in=85]  node [xshift=-0.2cm,yshift=-0.15cm] {$4|2$} node [xshift=0.5cm,yshift=-0.45cm] {$5|3$} (q_5)
    (q_1) edge  node [xshift=-0.5cm,yshift=0.5cm,swap]  {$4|2$} node [xshift=0.1cm,yshift=0.15cm,swap]  {$5|3$} (q_5)
          edge [in=50,out=40, loop] node [xshift=-0.8cm,yshift=-0.1cm,swap] {$0|4$} node [xshift=0cm,yshift=-0.4cm,swap] {$1|5$} ()
          edge  node [xshift=0cm,yshift=0.3cm,swap] {$2|1$} node [xshift=0cm,yshift=-0.3cm,swap]  {$3|0$} (q_3)
    (q_3) edge  node [xshift=0cm, yshift=0.15cm] {$1|4$}  node [xshift=0.5cm, yshift=-0.2cm] {$0|5$} (q_0)
          edge node [xshift=0.15cm, yshift=-0.1cm]  {$4|2$ $5|3$} (q_5) 
          edge [in=330,out=320, loop] node [xshift=-0.5cm,yshift=0.05cm,swap] {$2|1$} node [xshift=0.1cm, yshift=0.05cm,swap] {$3|0$} ()
    (q_5) edge [in=330,out=135]  node [xshift=-0.2cm,yshift=0cm,swap] {$1|5$} node [xshift=0.2cm, yshift=-0.4cm,swap] {$0|4$} (q_1)
          edge [out=185,in=355] node [xshift=0.1cm,yshift=0.05cm] {$2|0$ $3|1$} (q_3) 
          edge [in=330,out=320, loop] node [xshift=-0.5cm, yshift=0.05cm,swap] {$4|2$} node [xshift=0.2cm, yshift=0.05cm,swap] {$5|3$} ();
\end{tikzpicture}
}
\end{center}
\caption{An element $H \in \hn{6}$ of order $6$.}\label{fig:ExampleACallup}
\end{figure}
 
The level 2 dual has 36 states but only $6$ up-to $\omega$-equivalence. This can be verified either by direct calculation or in GAP \cite{GAP4} using  the AutomGrp package \cite{AutomGrp1.3.2} and its function ``MinimizationOfAutomaton( )'', which returns the minimal $\omega$-equivalent automaton to the input. Figure \ref{fig:zeroOfDual} shows the automaton $\overline{\dual{H}_{2}}$. This is in fact the zero of the semigroup $\{\dual{H}_{k}: k \in \N_{1}\}$).

\begin{figure}[htb]
\begin{center}
\scalebox{0.85}{
\begin{tikzpicture}[shorten >=0.5pt,node distance=3cm,on grid,auto] 
  \begin{scope}[xshift=0cm]
      \node[state] (q_0)   {$q_0$}; 
   \node[state] (q_1) [below left=of q_0] {$q_1$}; 
   \node[state] (q_2) [below right=of q_0] {$q_2$}; 
    \path[->] 
    (q_0) edge [bend right] node [swap] {$\ast|a_1$} (q_1)
    (q_1) edge [bend right] node [swap]  {$\ast|a_3 $} (q_2)
    (q_2) edge [bend right] node [swap] {$\ast|a_2 $} (q_0);
  \end{scope}
  \begin{scope}[xshift=6cm]
      \node[state] (q_0)   {$p_0$}; 
   \node[state] (q_1) [below left=of q_0] {$p_1$}; 
   \node[state] (q_2) [below right=of q_0] {$p_2$}; 
    \path[->] 
    (q_0) edge [bend right] node [swap] {$\ast|a_0$} (q_1)
    (q_1) edge [bend right] node [swap]  {$\ast|a_3 $} (q_2)
    (q_2) edge [bend right] node [swap] {$\ast|a_2 $} (q_0);
  \end{scope}
   
\end{tikzpicture}
}
\end{center}
\caption{The level 2 dual of $H$.}
\label{fig:zeroOfDual}
\end{figure}
\end{example}

The states $\{q_0,q_1, q_2, p_0, p_1, p_2\}$ of $\overline{\dual{H}_2}$ partition the $X_{6}^2$ as in Table~\ref{tab:Reduceddual}. One way to see this is to notice that the action of $H$ on the sets is as depicted in Figure~\ref{fig:zeroOfDual}.  Using Table~$\ref{tab:Reduceddual}$, it is straightforward to construct the automaton $A:=\Duak{H}{2}$ which is depicted in Figure~\ref{fig:Exmdigraph} (note that each drawn edge represents two edges with labels as listed). The automorphism $\phi_H$ of the underlying digraph of $A$ can be determined using Theorem~\ref{thm:representations}; in practise it can be more easily constructed using the observation that $\mathscr{H}(A, \phi_H)$ has minimal representative $H$. In particular, on the vertices on $G_A$, $\phi_H$ induces  the permutation which in cycle notation is 
$$ (p_0 \; p_1 \; p_2 ) (q_0 \; q_1\; q_2 ).$$
 We refer to the vertices $q_0, q_1, q_2$ as the vertices of the ``inner triangle'' and the vertices $p_0, p_1, p_2$ as the vertices of the ``outer triangle''.
 
\begin{table}[H]
\begin{align*}
q_0&=\{00,01,10,11,40,41,50,51\} \qquad \quad p_0=\{20,21,30,31\}\\
q_1&=\{24,25,34,35,44,45,54,55\} \qquad \quad p_1=\{04,05,15,15\} \\
q_2&=\{02,03,12,13,22,23,32,33\} \qquad \quad p_2=\{42,43,52,53\}.
\end{align*}
\caption{Partition of $X_6^2$ induced by states of 
$\dual{H}_{2}$.} \label{tab:Reduceddual}
\end{table}
\noindent %
\begin{figure}[htb]
\begin{center}
\scalebox{0.8}{
\begin{tikzpicture}[shorten >=0.5pt,node distance=3cm,on grid,auto] 
\begin{scope}[scale=0.8]
    \node[state] (q_0)   {$p_0$}; 
    \node[state] (q_1) [xshift=4cm,yshift=0cm] {$q_0$};
   \node[state] (q_3) [xshift=4cm,yshift=-3cm] {$q_2$}; 
   \node[state] (q_5)[xshift=8cm,yshift=-3cm] {$q_1$}; 
   \node[state] (q_4)[xshift=8cm,yshift=0cm] {$p_1$};
   \node[state] (q_2)[xshift=8cm,yshift=-6cm] {$p_2$};
    \path[->] 
    (q_0) edge node {$1$ \quad $0$} (q_1)
          edge [out=330,in=135]  node [xshift=-0.3cm,yshift=0.cm] {$2$} node  [xshift=0.3cm,yshift=-0.4cm] {$3$} (q_3)
          edge [out=90,in=90]  node [xshift=-0.2cm,yshift=0.15cm] {$4$} node [xshift=0.5cm,yshift=0.15cm] {$5$} (q_4)
    (q_1) edge [in=170,out=10]  node [xshift=-0.5cm,yshift=0.5cm,swap]  {$4$} node [xshift=0.1cm,yshift=0.5cm,swap]  {$5$} (q_4)
          edge [in=95,out=85, loop] node [xshift=-0.4cm,yshift=-0.3cm,swap] {$0$} node [xshift=0.4cm,yshift=-0.3cm,swap] {$1$} ()
          edge  node [xshift=0cm,yshift=0.3cm,swap] {$2$} node [xshift=0cm,yshift=-0.3cm,swap]  {$3$} (q_3)
    (q_3) edge  node [xshift=0cm, yshift=0.15cm] {$1$}  node [xshift=0.5cm, yshift=-0.2cm] {$0$} (q_0)
          edge node [xshift=0.15cm, yshift=-0.1cm]  {$4$ \quad $5$} (q_5) 
          edge [in=275,out=265, loop] node [xshift=-0.3cm,yshift=0.2cm,swap] {$2$} node [xshift=0.3cm, yshift=0.2cm,swap] {$3$} ()
    (q_5) edge [in=330,out=135]  node [xshift=-0.2cm,yshift=0cm,swap] {$1$} node [xshift=0.2cm, yshift=-0.4cm,swap] {$0$} (q_1)
          edge [out=260,in=100] node [xshift=0cm,yshift=-0.4cm,swap] {$2$} node [xshift=0cm,yshift=0.4cm,swap] {$3$} (q_2) 
          edge [in=355,out=5, loop] node [xshift=0cm, yshift=0.3cm,swap] {$4$} node [xshift=0cm, yshift=-0.3cm,swap] {$5$} ()
    (q_4) edge [in=0,out=0]  node [xshift=0cm,yshift=-0.4cm] {$2$} node [xshift=0cm, yshift=0.4cm] {$3$} (q_2)
          edge  node [xshift=0cm,yshift=-0.4cm] {$5$} node [xshift=0cm,yshift=0.4cm] {$4$} (q_5) 
          edge node [xshift=-0.5cm, yshift=0.05cm] {$1$} node [xshift=0.2cm, yshift=0.05cm] {$0$} (q_1)
    (q_2) edge [in=225,out=225]  node [xshift=-0.4cm,yshift=0.2cm] {$0$} node [xshift=0.4cm, yshift=-0.2cm] {$1$} (q_0)
          edge  node [xshift=-0.4cm,yshift=0.33cm] {$3$} node [xshift=0.4cm,yshift=-0.1cm] {$2$} (q_3) 
          edge node [xshift=0cm, yshift=-0.4cm,swap] {$4$} node [xshift=0cm, yshift=0.4cm,swap] {$5$} (q_5);
\end{scope}
\end{tikzpicture}
}
\end{center}
\caption{The automaton $\Duak{H}{2}$}\label{fig:Exmdigraph}
\end{figure}
Notice that both the domain and range automaton of $H$ are foldings of $A$. This phenomenon generalises, that is, for a synchronizing transducer $H$ representing an element of $\hn{n}$ of finite order, both the domain and range automata of $H$ are foldings of $\duA{H}$. 

We claim that $H$ is an element of order $6$, and in fact, in its action on the Cantor space $X_6^{-\N}$, all points are on orbits of length $6$.  It is enough to see that every circuit of $G_A$ is on orbit of length $6$ (any sufficiently long path must then be on an orbit of length $6$). This can be seen as follows.  A circuit of $G_A$ which is not formed by repeating the circuit (or a cyclic rotation of it) $p_0 \to p_1 \to p_2 \to p_0$ a finite number of times, must have an edge leaving a vertex in the inner triangle --- any such edge has orbit length $6$, and any circuit $p_0 \to p_1 \to p_2 \to p_0$ also has orbit length $6$. Therefore $H$ satisfies one of the equivalent conditions in the statement of Theorem~\ref{thm:mainresultIntro}.

We revisit this example throughout the course of this article (Sections~\ref{sec:application1}, \ref{Sec:Shadowstates} and \ref{sec:conclusion}), indeed after every major construction,  at the conclusion of which we will have demonstrated that $H$ is conjugate to a $6$-cycle.
$\bigcirc$

\subsubsection{\texorpdfstring{$\Duak{H}{k}$}{Lg} is minimal}

In this section, we prove that $\duA{H}$ is the unique (up-to automata isomorphism) minimal support for $H$. By minimal we mean that there is no synchronizing automaton $A$ which is a folding of $\duA{A}$ and supports $H$.

Our first step is the following useful lemma about the automaton $\duA{H}$.  It says that if two states $[\delta]$ and $[\gamma]$ are distinct in $\duA{H}$ but their two transition functions are the same, then by following the cycles of the level $k$ dual (by iteratively acting by $H$), we will eventually get to a pair of states which have different output letters in the level $k$ dual, and at that pair of locations, the states of $\duA{H}$ will still transition the same way, but the output functions of $\mathscr{H}(\duA{H},\phi_H)$ at these states will disagree on $\xn$.

\begin{lemma}\label{Lemma:findingdisagreement}
    Let $H\in \hn{n}$ be finite order.
	Let $\gamma, \delta \in \xn^{k}$ be such that the states $[\gamma], [\delta]$ of $\duA{A}$ are distinct. Suppose moreover that the maps $\pi_{\duA{H}}(\cdot, [\gamma])$ and $\pi_{\duA{H}}(\cdot, [\delta])$ coincide. 
	Then there is a natural $i$ with $0\leq i<o(H)$ and $x, y, y' \in \xn$ such that $y \ne y'$, $\pi_{\duA{H}}(\cdot, [\gamma]H^{i}) = \pi_{\duA{H}}(\cdot, [\delta]H^{i})$ but $H$ maps the edges $$([\gamma]H^{i},x, \pi_{\duA{H}}(x, [\gamma]H^{i})) \textrm{ and } ([\delta]H^{i},x, \pi_{\duA{H}}(x, [\delta]H^{i}))$$ respectively to the edges $$([\gamma]H^{i+1},y, \pi_{\duA{H}}(y, [\gamma]H^{i+1})) \textrm{ and }([\delta]H^{i+1},y', \pi_{\duA{H}}(y', [\delta]H^{i+1})).$$ 
	\end{lemma} 
	\begin{proof}
	Let $w = W([\gamma])$ and $v=W([\delta])$. Since $[\gamma] \ne [\delta]$, we may find words $u, w_2, v_2 \in Q_{H}^{\ast}$ and letters  $t \ne t' \in Q_{H}$ such that $w= u t w_2$ and $v = u t' v_2$. Set $i-1 := |u|$. 
	We note that for any $ j \in \N$ with $j \le i-1$, a straightforward induction argument shows, the edges $([\gamma], a, \pi_{\duA{H}}(a, [\gamma])H)$ and $([\delta], a, \pi_{\duA{H}}(q, [\delta]))$ map respectively under $H^{j}$ to edges $([\gamma]H^{j}, b, \pi_{\duA{H}}(a, [\gamma])H^{j})$ and $([\delta]H^{j}, b, \pi_{\duA{H}}(q, [\delta])H^{j})$, where $b$ = $\lambda_{H^{j}}(a, u_{[1,j]})$ (if $j=0$, take $b = a$). In particular it follows that $\pi_{H}(\cdot, t) = \pi_{H}(\cdot, t')$ and so, since $t \ne t'$, there is an   $a \in \xn$ be such that $y:=\lambda_{H}(a, t) \ne \lambda_{H}(a, t') = y'$. Let $x \in \xn$ be such that $\lambda_{H^{i-1}}(x, u) = a$. 
	Then it follows that the edges $$([\gamma]H^{i},x, \pi_{\duA{H}}(x, [\gamma]H^{i})) \mbox{ and }  ([\delta]H^{i},x, \pi_{\duA{H}}(x, [\delta]H^{i}))$$ are mapped respectively under $H$, to the edges $$([\gamma]H^{i+1},y, \pi_{\duA{H}}(y, [\gamma]H^{i+1})) \mbox{ and } ([\delta]H^{i+1},y', \pi_{\duA{H}}(y', [\delta]H^{i+1})).$$
\end{proof}

In the next lemma, we show that there is precisely one minimal automaton (up to isomorphism of automata) supporting a finite order element $H$ of $\hn{n}.$ 

\begin{prop}\label{Lemma:minimalityofdualaut}
	Let $H \in \hn{n}$ be an element of finite order. Then (up to isomorphism of automata) $\duA{H}$ is the minimal synchronizing automaton supporting $H$.  Furthermore, $\phi$ is the automorphism $\phi_H$ of Theorem \ref{thm:representations}.
\end{prop}
\begin{proof}
	Let $H\in \hn{n}$ be finite order of order $o(H)$.  We note that by results in \cite{OlukoyaOrder, BleakCameronOlukoyaI} $k$ is minimal such that all of the elements $H, H^{2}, \ldots, H^{o(H)-1}$ are synchronizing at level $k$ ($H^{i}$ being the product in $\hn{n}$ of $H$ with itself $i$ times).
	
	It follows from Theorem \ref{thm:representations} that $\duA{H}$ is an automaton supporting $H$ and indeed that $(\duA{H},\phi_H)$ induces $H$.  We argue below that $\duA{H}$ is a minimal such automaton: it is (automata isomorphic to) a folding of any synchronizing support of $H$.

	Suppose $A$ is a minimal support of $H$ realised by an automorphism $\psi_H$ of $G_A$.  We note that the minimal synchronizing level $l$ of $A$ is greater than or equal to $k$ since $\mathscr{H}(A, \psi_{H}^{i})$ is synchronizing at level $l$ and has minimal representative $H^{i}$.

	Suppose for a contradiction that $A$ is not equal  (i.e isomorphic as an automaton) to $\duA{H}$. 
	There are two cases. 
	
	Firstly for any state $q \in Q_{A}$, there is a state $p \in Q_{H}$ such that the set $W(q,l)$ of words of length $l$ which force $q$ is contained in the set $W(p,l)$ of words of length $l$ which force the state $p$ of $\duA{H}$. In this case, one observes that $\duA{H}$ is (isomorphic to) a folding of $A$ contradicting the minimality of $A$.

	The other case is the negation of the first. We assume that there is a pair of words $\gamma, \delta \in  \xn^{l}$ such that the state of $A$ forced by $\gamma$ is the same as the state of $A$ forced by  $\delta$ but $\gamma$ and $\delta$ force different states $[\gamma], [\delta]$ of $\duA{H}$. We may further assume that the states $[\gamma], [\delta]$ of $\duA{H}$ also satisfy $\pi_{\duA{H}}(\cdot, [\gamma]) = \pi_{\duA{H}}(\cdot, [\delta])$. This is because if $\pi_{\duA{H}}(\cdot, [\gamma]) \ne \pi_{\duA{H}}(\cdot, [\delta])$, then we may find a word $\nu \in \xnp$ such that $[\gamma'] := \pi_{\duA{H}}(\nu, [\gamma]) \ne \pi_{\duA{H}}(\nu, [\delta]) = [\delta']$, satisfy $\pi_{\duA{H}}(\cdot, [\gamma']) = \pi_{\duA{H}}(\cdot, [\delta]')$. Thus $\gamma\nu$ and $\delta \nu$ force the same state of $A$ but force, respectively, the states  $[\gamma']$ and $[\delta']$ of $\duA{H}$. We therefore replace $\gamma, \delta$ with $\gamma', \delta'$.
	
	Let $z_1$ be the state of $A$ forced by $\gamma$ and $\delta$ and let $z_1, z_2, \ldots, z_{o(H)}$ be the orbit of $z_1$ under the action $\psi_H$ of $H$. As $\mathscr{H}(A, \psi_{H}) = \mathscr{H}(\duA{H}, \phi_{H})$,  it must be the case that if $a, b \in \xn$ are such that the edge $([\gamma],a, [\Gamma])$ maps to $(([\gamma])H,b, ([\Gamma])H)$, then the edge $([\delta],a, [\Gamma])$ also maps to $(([\delta])H,b, ([\Gamma])H)$. Thus we conclude that $\pi_{\duA{H}}(\cdot,([\gamma])H) = \pi_{\duA{H}}(\cdot, ([\delta])H)$. Now observe that since $\gamma$ and $\delta$ are representatives of $[\gamma]$ and $[\delta]$, respectively, and since for any $q \in Q_{A}$, the state of $A$ forced by $\lambda_{H}(\gamma, q)$ is equal to the state of $A$ forced by $\lambda_{H}(\delta, q)$ is equal to $z_2$, it follows that there are  representatives of $([\gamma])H$ and $([\delta])H$ respectively such that the states of $A$ forced by these representative is $z_2$. We may thus repeat the argument in the $z_1$ case. By induction we therefore see that for any $1 \le i \le o(H)$,  the points $([\gamma])H^{i}$ and $([\delta])H^{i}$ satisfy the following:
 \begin{itemize}
     \item $\pi_{\duA{H}}(\cdot, ([\gamma])H^{i} ) = \pi_{\duA{H}}(\cdot, ([\delta])H^{i})$, and,
     \item whenever there are  $a,b \in \xn$, such that $(([\gamma])H^{i}, a, \nu)$ is an edge mapping under $H$ to the edge $(([\gamma])H^{i+1}, b (\nu)H$, then the edge $(([\delta])H^{i}),a, \nu)$ also maps under $H$ to $(([\delta])H^{i+1}),b,(\nu)H)$.
 \end{itemize}
  This contradicts Lemma~\ref{Lemma:findingdisagreement}. 
\end{proof}

\subsubsection{Building \texorpdfstring{$\duA{H}$}{Lg} on the fly}

Having established the uniqueness of $\duA{H}$, we introduce the following ad hoc method for constructing it.

Let $H \in \hn{n}$. Build an automaton $A$ inductively as follows:

\begin{enumerate}[label=Step \arabic*.]
    \item Set $Q_{A} = \{ q \in Q_A^{\ast} : q \text{ is prime and } \exists x \in \xn, \pi_{H^{|q|}}(x,q) = q\}$;
    \item For each $q \in Q_{A}$ do the following. For $x \in X_n$ let $m \in \N$ be the length of the cycle containing $x$ in the permutation $\lambda_{H^{|q|}}(\ast, q): \xn \to \xn$. Set $p' := \pi_{H^{m|q|}}(x, q^{m})$. Let $p \in Q_H^{\ast}$ be the prime word such that $p'$ is a power of $p$. If $p$ is not in $Q_{A}$ already, add it in. Add a transition $\pi_{A}(x, q) = p$.
    \item Repeat Step $2$.
\end{enumerate}

We have the following lemma.

\begin{lemma}
    Let $H \in \hn{n}$ be an element of finite order and let $A$ be the automaton constructed as above. Then $A$  is isomorphic as an automaton to $\duA{A}$. Indeed the map $\iota: Q_{\duA{H}} \to Q_{A}$ by $[\gamma] \mapsto W([\gamma])$ induces an automaton isomorphism from $\duA{H}$ to $A$.
\end{lemma}
\begin{proof}

We note that for each $x \in X$, $W[x^{k}]$ are precisely the states of $A$  generated in the first step of the iterative process above. The states added  are precisely the elements of  $\{W(\pi_{\duA{H}}(x, [y^{k}]: x,y \in \xn\} \backslash \{W[x^{k}]: x \in \xn\}$. The process then repeats for each iterative step in the construction of $A$.
\end{proof}

\begin{rem}
    We note that by the above lemma $H \in \hn{n}$ has finite order if and only if the iterative process described above for constructing $A$ terminates.
\end{rem}

\section{Water for the witch -- shrinking conjugacy class representatives}
Suppose we have a finite order element $H\in \hn{n}$, induced by $(A,\phi_H)$ for some synchronizing automaton $A$.  Under certain conditions we may employ a two-step process to find a new element $I\in \hn{n}$, where $I$ is conjugate to $H$, and $I$ is induced by $(B,\psi)$ for some synchronizing automaton $B$ with fewer states than $A$. In the case that $H$ has a free action on $\xnN$ and is conjugate to a product of  $N$-cycles this process will eventually result in single state transducer.

The first step in this process is a ``relabelling'' strategy that maximises collapses in $B$.

\subsection{{Relabellings and automata sequences}}

In this section we describe our relabelling  process. We want to be able to relabel a synchronizing automaton supporting a finite order element of $\hn{n}$ while maintaining some control on the collapse structure of the automaton. The key idea essentially appears in the decomposition algorithm of \cite{BleakCameronOlukoyaI}. That algorithm can be  interpreted as stating that two synchronizing automaton have isomorphic digraphs if there is a relabelling process, induced by successive foldings of the first automaton, which transforms it into the second automaton. Thus, the relabellings we describe 
arise from carefully choosing appropriate foldings and are dependent on enumerating all maximal (in a sense we make precise below) foldings that lie between an automaton and one of its foldings. These are the automaton sequences we discuss below.

\subsubsection{Synchronizing sequences and collapse chains}

Let $A = (X_{n}, Q_{A}, \pi_{A})$ be an automaton. Define an equivalence relation $\sim_{A}$ on the states of $A$ by $p \sim_{A} q$ if and only if the maps $\pi_{A}(\cdot, p): \xn \to Q_{A}$ and $\pi_{A}(\cdot, q): \xn \to Q_{A}$ are equal. For a state  $q \in Q_{A}$ let $\mathsf{q}$ represent the equivalence class of $q$ under $\sim_{A}$. Further set $\mathsf{Q}_{\mathsf{A}}:= \{ \mathsf{q} \mid q \in Q_{A}\}$ and let ${\pi}_{\mathsf{A}}: \mathsf{Q}_{\mathsf{A}} \to \mathsf{Q}_{\mathsf{A}}$ be defined by ${\pi}_{\mathsf{A}}(x, \mathsf{q}) = \mathsf{p}$ where $p = \pi_{A}(x,q)$. Observe that $\pi_{\mathsf{A}}$ is a well defined map. Define a new automaton  $\mathsf{A} = (X_{n}, \mathsf{Q}_{\mathsf{A}}, \pi_{\mathsf{A}})$ noting that $|\mathsf{Q}_{\mathsf{A}}| \le |Q_{A}|$ and $|\mathsf{Q}_{\mathsf{A}}| = |Q_{A}|$ implies that $A$ is isomorphic to $\mathsf{A}$.

Given an automaton $A$, let $A_{0}:=A, A_1, A_2, \ldots$ be the sequence of automata such that $A_{i} = \mathsf{A}_{i-1}$ for all $i \ge 1$. We call the sequence $(A_i)_{i \in \mathbb{N}}$ the \emph{synchronizing sequence of $A$}. We make a few observations. 

By definition each term in the synchronizing sequence is a folding of the automaton which precedes it, therefore there is a $j \in \mathbb{N}$ such that all the $A_{i}$ for $i \ge j$ are isomorphic to one another. By a simple induction argument, for each $i$, the states of $A_{i}$ corresponds to a partition of $Q_{A}$. We identify the states of $A_{i}$ with this partition. For two states  $q,p \in Q_{A}$ that belong to a state $P$ of $A_{i}$, $\pi_{A}(x,q)$ and $\pi_{A}(x,p)$ belong to the same state of $Q_{A_i}$ for all $x \in X_{n}$. We will use the language `\textit{two states of $A$ are identified at level $i$}' if the two named states belong to the same element of $Q_{A_{i}}$. 

If the automaton $A$ is synchronizing and core, then an easy induction argument shows that all the terms in its synchronizing sequence are core and  synchronizing as well (since they are all foldings of $A$). For example if $A = G(n,m)$, then the first $m$ terms of the synchronizing sequence of $A$ are $(G(n,m), G(n,m-1), G(n,m-2),\ldots, G(n,1)$, after this all the terms in the sequence are the single state automaton on $X_{n}$.

The result below is from \cite{AutGnr}.

\begin{theorem}\label{thm:collapsingprocedure}
	Let  $A$ be an automaton and $A_{0}:=A, A_1, A_2, \ldots$ be the synchronizing sequence of $A$. Then 
	\begin{enumerate}
		\item a pair of states $p,q \in Q_{A}$, belong to the same element $t \in Q_{A_i}$ if and only if for all words $a \in X_{n}^{i}$, $\pi_{A}(a,p) = \pi_{A}(a, q)$, and
		
		\item $A$ is synchronizing if and only if there is a $j \in \mathbb{N}$ such that $|Q_{A_{j}}| = 1$. The minimal $j$ for which $|A_{j}| = 1$ is the minimal synchronizing level of $A$.
			\end{enumerate}
	
\end{theorem}

We also require the notion of a \textit{collapse chain} from \cite{BleakCameronOlukoyaI}. 
Let $A$ and $B$ be automata. Let $A = A_0, A_1, \ldots, A_{k} = B$  be a sequence such that $A_{i+1}$ is obtained from $A_{i}$ by identifying pairs of states $p \sim_{A_{i}} q$.
We note that as distinct from the synchronizing sequence, we do not necessarily make all possible identifications. 
Such a sequence is called a \textit{collapse chain} if at each step, we make the maximum number of collapses possible relative to the final automaton $B$. 
That is, for $u, v \in Q_{A}$  belonging to the same state of $B$, in the minimal  $A_{i}$  such that $[u] \sim_{A_{i}} [v]$, we have $[u] = [v]$ in $A_{i+1}$. 
We note that this condition means that a collapse chain is unique. 
Moreover, if  $A = A_0, A_1, \ldots, A_{k} = B$ is a collapse chain, then for any $0 \le j \le k$, $A = A_0, A_1, \ldots, A_{j}$ and  $A_{j}, \ldots, A_{k}$ are also collapse chains.

We say that an automaton $B$ \textit{belongs to a collapse chain of} $A$ if there is a collapse chain $A = A_0, A_1, \ldots, A_{k} = B$. In this case, we call the collapse chain $A = A_0, A_1, \ldots, A_{k} = B$, the \textit{collapse chain from $A$ to $B$}. If $B$ is a single state automaton, the collapse chain from $A$ to $B$ is precisely the synchronizing sequence of $A$. Thus a collapse chain can be thought of as a synchronizing sequence relative to its end point. 

It is not hard to see that if an automaton $B$ is a folding of a synchronizing automaton $A$, then there is a collapse chain from $A$ to $B$. Therefore, a synchronizing automaton $B$ belongs to a collapse chain of $A$ if and only if $B$ is a folding of $A$. In particular if $B$ belongs to a collapse chain of $A$, then $B$ is synchronizing at the minimal synchronizing level of $A$.

The following result about collapse chains is proved similarly to Theorem~\ref{thm:collapsingprocedure}.

\begin{theorem}\label{thm:collapsechains}
	Let  $A$ be an automaton. Suppose an automaton $B$ belongs to a collapse chain   $A_{0}:=A, A_1, A_2, \ldots, A_{m} = B$ of $A$.  Then a pair of states $p,q \in Q_{A}$ belong to the same element $t \in Q_{A_i}$ if and only if $p,q$ belong to the same state of $Q_{B}$ and for all words $a \in X_{n}^{i}$, $\pi_{A}(a,p) = \pi_{A}(a, q)$.
\end{theorem}

\subsubsection{Relabellings via foldings and conjugacy}

\begin{predef}\label{def:relabelling}
	Let $A$ be an automaton and $A= A_{0},A_{1}, \ldots, A_{m}$ be a collapse chain of $A$. Let $0 \le k \le m$ and $\rho_{k}$ be a vertex fixing automorphism of $G_{A_{k}}$. Define $A'$ to be the automaton with $Q_{A'} = Q_A$ and transition function defined as follows: for $p \in Q_{A}$, set $\pi_{A'}(x', p) = q$ if and only if there is an $x \in \xn$ such that $\pi_{A}(x,p) = q$ with $\lambda_{\mathscr{H}(A_{k}, \rho_{k})}(x, [p]) = x'$. We call $A'$ the \emph{relabelling of $A$ by $(A_{k}, \rho_{k})$} or \emph{the relabelling of $A$ by (the transducer) $\mathscr{H}(A_{k}, \rho_{k})$}.
\end{predef}

Note that if we relabel $A$ by $(A_k, \rho_{k})$, then the resulting automaton $A'$ is \textit{isomorphic} to $A$ in the sense that there is a natural digraph isomorphism from the underlying digraph of $A'$ to the underlying digraph of $A$ that fixes states and which maps the relabelled edges of $A'$ to the original edges in $A$. More precisely, if $(p,x,q)$ is an edge of $G_{A}$ and $\lambda_{\mathscr{H}(A_{k}, \rho_{k})}(x, [p]) = x'$, then the natural digraph isomorphism maps the edge $(p,x',q)$ of $A'$ to the edge $(p,x,q)$ of $A$.  The point of view one should have in mind is that we have renamed/relabelled the edges of $A$ by switching edge labels on edges which are parallel edges in $A_{k}$. Notice that if we relabel by $(A_0,\rho_0)$, then all we do is switch labels on parallel edges in $A$, thus the resulting underlying digraph would not change, but a ``fixed'' drawing of it would be relabelled. 

The following results says that a relabelling of an automaton $A$ by $(A_k, \rho_k)$ essentially preserves the collapse structure of $A$.

\begin{lemma}\label{lemma:relabelling}
	Let $A$ be an automaton and $A= A_0, A_{1}, \ldots, A_{m}$ be a collapse chain for $A$. Let $0 \le k \le m$ and $\phi$ be a vertex fixing automorphism of $G_{A_{k}}$. Let $A'$ be the relabelling of $A$ by $(A_{k}, \phi)$. Then 
 \begin{itemize}
     \item the underlying digraphs of $A'$ and $A$ are isomorphic;
     \item $A_{m}$ belongs to a collapse chain of $A'$;
     \item if $A'= A'_0, A'_{1}, \ldots, A'_{l}$  is the collapse chain from $A'$ to $A_{m}$ then,
     \begin{itemize}
         \item then $l \le m$, and,
         \item two states of $u,v$ of $A$ belong to the same state of $A_{i}$ if and only if  $u$ and $v$ belong to the same state of $A'_{i'}$ for some $i' \le i$; in other words, the partition of $Q_A = Q_{A'}$ induced by $A'_i$ is coarser than the partition induced by $A_i$.
     \end{itemize}
 \end{itemize}
 
\end{lemma}
\begin{proof}
     We may consider $A$ as a non-minimal  transducer where each state induces the identity transformation of the set $\xn$. Consider the  product $A \ast \mathscr{H}(A_{k}, \phi)$. Let $x \in \xn$ and $p,q \in Q_{A}$ such that $\pi_{A}(x, p) = q$. Then we have, $\pi_{A \ast \mathscr{H}(A_{k}, \phi)}(x, (p,[p])) = (q,[q])$ and $\lambda_{A \ast \mathscr{H}(A_{k}, \phi)}(x, (p,[p])) = \lambda_{\mathscr{H}(A_{k}, \phi)}(x, [p])$.  Let $D$ be the  subtransducer of $A \ast \mathscr{H}(A_{k}, \phi)$ with state set $Q_{D} = \{ (p, [p]) \mid p \in Q_{A} \}$. So $\pi_{D}$ and $\lambda_D$ are restrictions of $\pi_{A \ast \mathscr{H}(A_{k}, \phi)}$ and $\lambda_{D}$ to   $\xn \times Q_D$.
     Setting $A'$ to be the range automaton of $D$ we see that $A'$ is the relabelling of $A$ by $(A_{k}, \phi)$. From this it follows that the underlying digraph of $A'$ is isomorphic to the underlying digraph of $A$.

    We consider the second and last bullet points together.

	Let $q, r$ be two states of $A$ which belong to the same state of $A_{m}$ and which transition identically on all words of length $j$ and suppose $j$ is minimal for which this happens. That is $j \in \N$ is minimal such that $\pi_{A}(\cdot, q)$ and $\pi_{A}(\cdot, p)$ restrict to the same map $\tau_{q,r}: \xn^{j} \to Q_{A}$ on $\xn^{j}$.  
    Observe that for an arbitrary state  $p \in Q_{A}$, by definition, $\tau_{q,r}^{-1}(p)$ is the set consisting of all words $\gamma$ such that $\pi_{A}(\gamma, q) = \pi_{A}(\gamma, r) = p$. 
    
    We consider two cases.
	
	First suppose that $k \ge j$.  This means that in $A_{k}$, the states $[q]$ and $[r]$ are equal. Thus, $\lambda_{\mathscr{H}(A_{k}, \phi)}(\gamma, [q]) = \lambda_{\mathscr{H}(A_{k}, \phi)}(\gamma, [r])$ for any $\gamma \in \xn^{*}$. Therefore in $A'$ we see that for any arbitrary state $p \in Q_A$ the set of words $\nu \in \xn^{j}$ for which $\pi_{A'}(\nu, q) =p$ or $\pi_{A'}(\nu, r) = p$ is precisely the set $\{\lambda_{\mathscr{H}(A_{k}, \phi)}(\gamma, [q]) \mid \gamma \in \tau_{q,r}^{-1}(p)\}$, and, by the preceding sentence,  $\pi_{A'}(\cdot, q)$ and $\pi_{A'}(\cdot, r)$ coincide on this set. Since $\sqcup_{p \in Q_{A}}\tau_{q,r}^{-1}(p) = \xn^{j}$, we conclude that $\pi_{A'}(\cdot, q)$ and $\pi_{A'}(\cdot, r)$ coincide on the set $\xn^{j}$.

    Note that this case also demonstrates that $A_k$ belongs to a collapse chain of $A'$ (and so $A_{j}$ for $j \ge k$ belongs to a collapse chain of $A'$). Since all the collapses from $A$ to $A_j$ can be replicated in $A'$

	Now suppose that $k < j$. 
 This means that the states $[q]$ and $[r]$ are distinct states of $A_{k}$ such that $\pi_{A_{k}}(\cdot, [q])$ and $\pi_{A_{k}}(\cdot, [r])$ coincide on $\xn^{j-k}$. 
 Let $p \in Q_A$ and $\gamma \in \tau_{q,r}^{-1}(p)$ be arbitrary. 
 Set $\gamma_{1}$ to be the length $j-k$ prefix of $\gamma$ and set $\gamma_{2} \in \xn^{k}$ such that $\gamma_{1} \gamma_2 = \gamma$. Set $[s] = \pi_{A_{k}}(\gamma_1, [q])=\pi_{A_{k}}(\gamma_1, [r])$ and set $\kappa \in \xn^{k}$ such that $\lambda_{\mathscr{H}(A_{k},\phi}(\kappa, [s]) = \gamma_2$.
 For $t \in \{q,r\}$, let $\delta_{t} \in \xn^{j-k}$ be defined such that $\lambda_{\mathscr{H}(A_{k}, \phi)}(\delta_{t}, [t]) = \gamma_{1}$. 
 Then, since $\phi$ is a vertex fixing automorphism of $A_{k}$, we notice that $\pi_{A}(\delta_{q}, q)$, $\pi_{A}(\delta_r, r)$, $\pi_{A}(\gamma_{1},t)$, $t \in \{q,r\}$, all belong to the same state of $A_{k}$. 
 This means that $\pi_{A}(\kappa,\pi_{A}(\delta_{q}, q)) = \pi_{A}(\kappa,\pi_{A}(\delta_{r}, r)) = p'$.
 (Note that $[p'] = [p]$ in $A_{k}$.)  
 Therefore, $\pi_{A\ast \mathscr{H}(A_{k}, \phi)}(\delta_{q}\kappa, (q,[q])) = (p',[p]) = \pi_{A\ast \mathscr{H}(A_{k}, \phi)}(\delta_{r}\kappa, (r,[r]))$. 
 Since $\lambda_{A\ast \mathscr{H}(A_{k}, \phi)}(\delta_{q}\kappa, (q,[q])) = \gamma = \lambda_{A\ast \mathscr{H}(A_{k}, \phi)}(\delta_{r}\kappa, (r,[r]))$, we see that in $A'$, $\pi_{A'}(\gamma, q) = p' =\pi_{A'}(\gamma, r)$. 
    Therefore,  $\pi_{A'}(\cdot, q)$ and $\pi_{A'}(\cdot, r)$ coincide on the set $\tau_{q,r}^{-1}(p)$. 
    Since $p$ was chosen arbitrarily and $\sqcup_{p \in Q_{A}}\tau_{q,r}^{-1}(p) = \xn^{j}$, we conclude that $\pi_{A'}(\cdot, q)$ and $\pi_{A'}(\cdot, r)$ coincide on the set $\xn^{j}$.
	 
	 The final bullet point now follows by Theorem~\ref{thm:collapsechains}.

\end{proof}

Let $A$ be a synchronizing automaton and $A'$ be the relabelling of $A$ by $(A_{k},\phi_{k})$. Set $\iota: G_{A} \to G_{A'}$ to be the natural digraph isomorphism. If $\varphi$ is an automorphism of the underlying digraph $G_{A}$ of $A$, then we will mean by the \textit{induced automorphism} $\varphi'$ of $G_{A'}$ precisely the map $\iota^{-1}\varphi\iota$.

\begin{lemma}\label{lemma:relabellinginducesconj}
 	Let $H \in \hn{n}$ be an element of finite order, and $A$ a support of $H$ realised by an automorphism $\phi_H$ of its underlying digraph. Let $B$ be a folding of $A$ and $\psi$ a vertex fixing automorphism of $B$. Let $A'$ be the relabelling of $A$ by $(B, \psi)$ and $\varphi$ be the induced isomorphism from the underlying digraph of $A$ to the underlying digraph of $A'$. Set $I$ to be the minimal representative of the transducer $\mathscr{H}(A,A', \varphi)$. Then $I^{-1}AI$ is the minimal representative of $\mathscr{H}(A', \phi_{H}^{\varphi})$.
 \end{lemma}
 \begin{proof}
 	This is a straight-forward application of the definitions.
 \end{proof}
 
 In the situation of Lemma \ref{lemma:relabellinginducesconj} we refer to the resulting transducer  $I^{-1}AI$ as the \emph{conjugate of  $H$ induced by the relabelling $A\mapsto A'$}.

\subsection{Applications of relabellings I: removing collapse obstructions}\label{sec:application1}

The first application we make of relabellings resolves the following natural scenario. Often, in a synchronizing automaton $A$ we have a pair of states $p$ and $q$ with the property that $\pi_{A}(\cdot, p)$ and $\pi_A(\cdot, q)$ do not coincide on $\xn$ yet, for any $r \in Q_A$, $|\{x \in \xn : \pi_A(x, p) = r\}| = |\{x \in \xn : \pi_{A}(x,q) = r\}$. In other words, the states $p$ and $q$ ought to be collapsible but are not due to a misalignment in their edge labels. Our first application of relabellings is to show that unfortunate edge labels can be overcome. 

The scenario we formalise below is a little nuanced than the one just described; that illustration however is emblematic and should be kept in the back of the reader's mind. 

\subsubsection{Characterising disagreement: constructing discriminant permutations \texorpdfstring{$\disc{s,t,Q}$}{Lg}}

Let $A$ be an automaton and let $s,t \in Q_{A}$.  Set the notation:
\begin{align*}
\edgesT{A}{s}{t}&\seteq\{(s,x,t)\in \textrm{E}_A\}; \text{ and}\\
\letsT{A}{s}{t}&\seteq\{x\in \xn\mid (s,x,t)\in \edgesT{A}{s}{t}\}.
\end{align*}

 We may leave out the explicit mention of the automaton $A$ when it is clear from context, writing simply E$(s,t)$ and Letters$(s,t)$ for these sets in this case.

 Let $Q\subseteq Q_{A}$ and $s\in Q_{A}$.  Set the notation 
 \[
 X_{s,Q} \seteq \bigsqcup_{p \in Q} \lets{s}{p}.
 \]
 
Now, suppose $s, t\in Q_A$ and suppose further there is a subset $Q \subseteq Q_{A}$ so that 

\begin{enumerate}
    \item $X_{s,Q} = X_{t,Q}$, and \label{disc:union}    \item for all $p \in Q$ we have $\letCnt{s}{p}=\letCnt{t}{p}$ \label{disc:cardinality}.
\end{enumerate} Then to describe this situation we say \textit{$s$ and $t$ distribute similarly over $Q$}.  (Note in passing that for some choices of $s$ and $t$ the only possible such set $Q$ may be empty.) For any states $s$ and $t$ and set $Q\subset Q_A$ so that $s$ and $t$ distribute similarly over $Q$, we denote by $X_Q$ the set $X_{s,Q}=X_{t,Q}$. We call $X_{Q} \subseteq \xn$ \textit{the agreement alphabet (of $s$ and $t$ on $Q$)} noting that if $Q = Q_{A}$, then $X_{Q} = \xn$.

For states $s,t \in Q_A$ which distribute similarly over a subset $Q$ of $Q_A$, define a bijection $\mathrm{disc}(s,t,Q): X_{Q} \to X_{Q}$ as follows.

First, let $p_1, \ldots, p_r \in Q$ 
be a maximal sequence of 
distinct states such that for $1 \le i \le r$ there is an $x \in X_{Q}$ with $\pi_{A}(x, s) = p_i$.
Observe that the sets
\[
\{\lets{s}{p_i}\mid 1\leq i\leq r\}
\] and
\[\{\lets{t}{p_i}\mid 1\leq i\leq r\}
\] form partitions of $X_{Q}$, with equal-size corresponding parts in index $i$.

Now, for $1 \le i \le r$, set $\mathrm{disc}(s,t,Q)$ to act as the 
identity on $\lets{s}{p_i} \cap \lets{t}{p_i}$. 
Set $$\lets{s}{p_i}':= \lets{s}{p_i} \backslash (\lets{s}{p_i} \cap \lets{t}{p_i})$$ and
$$\lets{t}{p_i}':= \lets{t}{p_i} \backslash (\lets{s}{p_i} \cap \lets{t}{p_i}).$$ 
We note that $|\lets{s}{p_i}'| = |{\lets{t}{p_i}'}|$ and indeed that $$Y_{Q}\seteq\bigcup_{1 \le i \le r} \lets{s}{p_i}' = \bigcup_{1 \le i \le r} \lets{t}{p_i}'.$$ Order the elements of $\lets{s}{p_i}'$ and $\lets{t}{p_i}'$ with the order induced from $\xn$. For $1 \le i \le r$ and $x \in \lets{s}{p_i}'$ we write $x'$ for the corresponding element of $\lets{t}{p_i}'$, that is, in the ordering of $\lets{s}{p_i}'$ and $\lets{t}{p_i}'$ induced from $\xn$, $x$ and $x'$ have the same index.

Using the definitions and facts above we extend the definition of $\disc{s,t,Q}$ over the set $Y_{Q}$  by the rule $x\mapsto x'$.  One easily checks that the resulting function 
$$\disc{s,t,Q}:X_Q\to X_Q$$
is a well-defined bijection.  Further, observe that for $x_0\in Y_{Q}$ the function $\disc{s,t,Q}$ contains a cycle $(x_0 \ x_1 \ x_2 \ \ldots \ x_{k-1}) $ in its disjoint cycle decomposition, where for all $i$ we have $x_{i+1}=x_i'$ (indices taken mod $k$).  Recall as well that $\disc{s,t,Q}$ acts as the identity over the set $X_Q\backslash Y_{Q}$.

For $s$ and $t$ satisfying points \eqref{disc:union} and \eqref{disc:cardinality} for some set $Q$ we call $\disc{s,t, Q}$ \textit{the discriminant of $s$ and $t$}; it is a permutation that encodes the difference in transitions between $s$ and $t$ amongst the states $Q$. In the case that $Q=Q_A$, we will write $\disc{s,t}$ for the bijection $\disc{s,t, Q_{A}}$. As with the notation $\lets{p}{q}$, we often run in to situations where we compute discriminant permutations in distinct  automata sharing the same state set, in such cases we use the notation $\discT{s,t,Q}{A}$ and $\discT{s,t}{A}$  to emphasise the automaton in which the permutation is calculated.

\subsubsection{Maximising agreement} \label{Sec:discandamalg}

Given a synchronizing automaton $A$ we want to maximise the number of collapses at each step in its synchronizing sequence, that is, at each step of its synchronizing sequence all states that ought to be collapsible should be collapsible.
The following lemma is the key tool we use to do this. It states that for a pair of states  $s,t$ of some synchronizing automaton $A$ which ought to be collapsible, there is a folding of $A$ in which the states containing $s$ and $t$ are distinct and the cycles of the discriminant permutation are parallel edges. We may then induce a relabelling of $A$ using a vertex fixing automorphism of this folding in order to make the states $s$ and $t$ actually collapsible.

\begin{lemma}\label{Lemma:parallelegdes}
    Let $A$ be an automaton, and let $s, t$ be distinct states of $A$. Let $Q \subseteq Q_{A}$ be any set of states over which $s$ and $t$ distribute similarly and let $X_{Q}$ be the agreement alphabet. Let $(A_i)_{ 1 \le i \le m}$ be a collapse chain of $A$  such that $s,t$ belong to the same state of $A_{m}$. Let $1 \le k < m$ be minimal such that $\pi_{A_{k}}(\cdot, [t])$ and $\pi_{A_{k}}(\cdot, [s])$ are equal on $X_{Q}$. Then for  $x,y \in X_{Q}$ which belong to the same disjoint cycle of $\disc{s,t,Q}$, $$\pi_{A_{k}}(x,  [s]) = \pi_{A_{k}}(y, [s]) = \pi_{A_{k}}(y, [t]) = \pi_{A_{k}}(x, [t]).$$
\end{lemma}
\begin{proof}
    By assumption $\pi_{A_{k}}(\cdot, [s]) = \pi_{A_{k}}(\cdot, [t])$.
    
    An easy induction argument using the definition of $\disc{s,t,Q}$ now shows that for any $x, y \in \xn$ such that $y$ belongs to the orbit of $x$ under the action of $\disc{s,t,Q}$, $$\pi_{A_{k}}(x, [s]) = \pi_{A_{k}}(y, [s])= \pi_{A_{k}}(x, [t]) = \pi_{A_{k}}(y, [t]).$$ This follows since for any $x \in \xn$, $\pi_{A}(x, s) = \pi_{A}((x)\disc{s,t,Q}, t)$.

\end{proof}

We formalise below what we mean by ``maximising collapses''. This is a property that is captured in the underlying digraph of the automaton.

 Let $A$ be an automaton and let $A = A_0, A_1, \ldots, A_m=:B$ be a collapse chain. Let $G:=G_{A}$ be the underlying digraph of $A$. Define  a sequence $G:= G_0, G_{1}, \ldots$ as follows. Assuming $G_{i}$ is defined and the vertices of $G_i$ induce a partition of $Q_A$, $G_{i+1}$ is obtained from $G_i$ in the following manner. Let $\sim$ be the equivalence relation on the vertices $Q_{G_{i}}$ of $G_{i}$ that relates two vertices $p,q$ if and only if the following two conditions hold:
 \begin{itemize}
     \item there is a state $[p,q] \in A_m$ such that for any $r \in Q_A$ belonging either to $p$ or $q$, $r \in [p,q]$;
     \item  for every vertex $t \in  Q_{G_{i}}$ the number of edges from $q$ to $t$ is precisely the number of edges from $p$ to $t$.
 \end{itemize}
 If $p\in Q_{G_i}$ write $[p]_{i+1}$ for the equivalence class of $p$ under the relation $\sim$.  Set $Q_{G_{i+1}} = \{ [p]_{i+1} \mid p \in Q_{G_{i}}\}$. Now suppose $p, q \in Q_{G_{i}}$ and enumerate those elements of $[q]_{i+1}$ which have an incoming edge from a vertex in $[p]_{i+1}$ in some order as $q_{1}, q_2, \ldots, q_{r}$. For $1 \le j \le r$, let $k_{j}$ be the number of edges from $p$ to $q_{j}$ and set  
$ec(i+1,p,q)\seteq\sum_{1 \le j \le r} k_{j}$. Set $G_{i+1}$ to be the directed graph with vertices $Q_{G_{i+1}}$ and with $ec(i+1,p,q)$ many edges from $[p]_{i+1}$ to $[q]_{i+1}$ for each $[p]_{i+1}$, $[q]_{i+1}\in Q_{G_{i+1}}$.

We refer to the resulting sequence $G_0$, $G_1$, \ldots, as defined above as the \emph{amalgamation sequence of $G$ relative to $G_{A}$}. It is unique up to the choice of edge identifications. If $A$ is a synchronizing automaton and $B$ is the single state automaton, then the sequence $G_0$, $G_1$, \ldots,  is the of amalgamation sequence of $G$ as defined in the literature (see \cite{Williams73} for example).  By construction, for each natural $i$ the states of $G_i$ induce a partition of the states of $A$. This partition is finer than the partition of $Q_A$ induced by  $B$ but coarser than the partition induced by $A_i$.  It follows that after finitely many steps, the amalgamation sequence stabilises to a fixed digraph which induces precisely the same partition of $Q_A$ as $A_m$.

We maximise the collapses of an automaton $A$ by  relabelling $A$ such that the successive partitions of $Q_A$ induced  by its synchronizing sequence coincides with those induced by the amalgamation sequence of $G_A$. One can think of the amalgamation sequence as describing the ``ideal'' route to the final term in the synchronising sequence.

\begin{lemma}\label{lemma:collapseadherence}
    Let $A$ be an automaton with underlying digraph $G$. Let  $A= A_0, A_{1}, \ldots, A_m$ be a collapse chain of $A$ and let $G= G_0, G_1, \ldots, G_{l}$ be the amalgamation sequence of $G$ relative to $G_{A_m}$. Then there is a relabelling $B$ of $A$ with a collapse chain $B = B_0, B_1, \ldots, B_l = A_m$ such that the underlying digraph of $B_{i}$ is $G_{i}$; in particular, the partition of the state set of $A$ induced by $B_i$ is the same partition induced by $G_{i}$.
\end{lemma} 
\begin{proof}
    We proceed by induction on the amalgamation sequence.
    
    We begin with the base case. Let $s,t \in Q_{A}$ be distinct such that $s$ and $t$ belong to the same state of $G_{1}$. This means that $s$ and $t$ distribute similarly over $Q_{A}$. Suppose that $\disc{s,t}$ is not trivial. 
    
    Let $k \in \N$ be minimal such that $s$ and $t$ belong to the same state of $A_{k+1}$. Note that since $\disc{s,t}$ is not trivial, then $k \ge 1$.  By Lemma~\ref{Lemma:parallelegdes}, for any $x,y$ which belong to the same orbit under $\disc{s,t}$, $$\pi_{A_{k}}(x, [s]) = \pi_{A_{k}}(y, [s]) = \pi_{A_{k}}(x, [t]) = \pi_{A_{k}}(x, [t]).$$
    
    Let $\lambda_{A_{k}}$ be defined such that $\lambda_{A_{k}}(\cdot, [q]): \xn \to \xn$ is trivial whenever $[q]$ is not equal to $[t]$. We set $\lambda_{A_{k}}(\cdot, [t]) = \disc{s,t}^{-1}$. We note that the transducer $A_{k}$ is induced by a vertex fixing automorphism of $A_{k}$. Furthermore, for any pair $(u,v) \ne (s,t)$ such that $u,v$ belong to the same state of $G_{1}$ and $\disc{u,v}$ is trivial, $[u] = [v]$ in $A_{k}$.
    
    Let $C$ be the relabelling of $A$ by the transducer $A_{k}$. Note that $Q_{C} = Q_{A}$ and $G$ remains the underlying digraph of $C$. It therefore follows that for any pair $u,v$ in $Q_A$ which distribute similarly over $Q_{A}$, $u,v$ still distribute similarly over $Q_{A}$ in $C$. If moreover, $\discT{u,v}{A}$ is trivial, then construction of $\lambda_{A_{k}}$, means $\discT{u,v}{C}$ remains trivial.  We note as well that $s,t$ distribute similarly over $Q_{C}$ in $C$, and $\discT{s,t}{C}$ is trivial. Lastly, by Lemma~\ref{lemma:relabelling}, $A_m$ belongs to a collapse chain of $C$.
    
    Applying an induction argument, there is an automaton $C$ a relabelling of $A$ such that for any pair $s,t \in Q_{A}$ which belong to the same state of $G_{1}$, $\discT{s,t}{C}$ is trivial. In particular, such $s,t$ satisfy, $\pi_{C}(\cdot, s) = \pi_{C}(\cdot, t)$.
    
    Now assume by induction that there is a relabelling $C$ of $A$ with a collapse chain $C = C_0, C_1, \ldots, C_{M} = A_{m}$ ($l \le M \le m)$ possessing the following property: for $0 \le i \le k < l$ two states $s,t \in Q_{A}$ which belong to the same state of $G_{i}$ belong to the same state of $C_{i}$.
    
    Let $s,t \in Q_{A}$ and suppose $s$ and $t$ belong to the same state of $G_{k+1}$ but do not belong to the same state of $C_{k+1}$.  We note that since the underlying digraph of $C_{k}$ is the same as $G_{k}$ and they induce the same partition of the state set $Q_{A}$, then $s$ and $t$ belong to distinct states of $C_{k}$ and so to distinct states of $G_{k}$. The fact that $s$ and $t$ belong to the same state of $G_{k+1}$ means that $[s]$ and $[t]$ distribute similarly over $Q_{C_k}$ but $\discT{[s],[t]}{C_{k}}$ is not trivial.
    
    Let $k< j \le l$ be minimal such that $s$ and $t$ belong to the same state of $C_{j+1}$. By Lemma~\ref{Lemma:parallelegdes} once more, we have the equalities: for any $x,y$ which belong to the same orbit under $\discT{[s],[t]}{C_{k}}$, $$\pi_{C_{j}}(x, [s]) = \pi_{C_{j}}(y, [s]) = \pi_{C_{j}}(x, [t]) = \pi_{C_{j}}(x, [t]).$$

    Let $\lambda_{C_{j}}$ be defined such that $\lambda_{C_{j}}(\cdot, [q]): \xn \to \xn$ is trivial whenever $[q]$ is not equal to $[t]$. We set $\lambda_{C_{j}}(\cdot, [t]) = \discT{[s],[t]}{C_{k}}^{-1}$. We note that the transducer $C_{j}$ is induced by a vertex fixing automorphism of $C_{j}$. Furthermore, for any pair $(u,v) \ne (s,t)$ such that $u,v$ belong to the same state of $G_{k+1}$ and $\discT{[u],[v]}{C_{k}}$ is trivial, $[u] = [v]$ in $C_{j}$.
    
    Let $D$ be the relabelling of $C$ by the transducer $C_{j}$. By Lemma~\ref{lemma:relabelling}, $A_m$ belongs to a collapse chain of $D$. Let $D:=D_0, D_{1}, \ldots, D_{L} := A_m $ be the collapse chain from $D$ to  $A_m$, noting that $l \le L \le M$.
    
    Let $u,v \in Q_{A}$ belong to the same state of $G_{i}$ for some $0 \le i \le k < l$. Then by the inductive assumption and Lemma~\ref{lemma:relabelling}, $u,v$ belong to the same state of $D_{i}$.
    
    Let $u,v \in Q_{A}$ belong to the same state of $G_{k+1}$ and suppose that $\discT{u,v}{C_{k}}$ is trivial. Note that since $\discT{u,v}{C_{k}}$  and $[u],[v]$ distribute similarly over $Q_{C_{k}}$, $[u] = [v]$ in $C_{k+1}$. Therefore, Lemma~\ref{lemma:relabelling} implies that $[u] = [v]$ in $D_{k+1}$ as well.
    
    Lastly observe that $[s] = [t]$ in $D_{k+1}$ by construction of $\lambda_{C_{k}}$ and the fact that states which are identified in $C_{k}$ remain identified in $D_{k+1}$. 
    
    The result now follows by induction.
\end{proof}

Let $A$ be an automaton, let $A= A_0, A_1, \ldots,A_m$ be a collapse chain of $A$ and $G_{A}=G_0, G_1,\ldots, G_{l}$ be the amalgamation sequence of $G_A$ relative to $A_m$. We say that $A$ is \textit{unimpeded relative to $A_m$} if the underlying digraph of $A_i$ is $G_i$. If, additionally, the collapse chain $A= A_0, A_1, \ldots,A_m$ is the synchronizing sequence of $A$ cut off at the point at which it stabilises, then we  say simply that \textit{$A$ is unimpeded}.

\begin{example}\label{exm:unimpededautomata}
The automaton in Figure~\ref{fig:Exmdigraph} and the de Bruijn graph $G(n,m)$ are examples of unimpeded automata. $\bigcirc$ 

\end{example}

The property of being unimpeded is invariant under relabellings induced by terms in the synchronizing sequence of an automaton but not necessarily by relabellings induced by terms in a general collapse chain.

\subsection{Applications of relabellings II: coordinating agreement along orbits}

This application of relabellings deals with the following scenario. We have an element of finite order $H \in \hn{n}$ which is induced by a pair $(A, \phi)$ where $A$ is an unimpeded support of $H$ realised by $\phi$. Suppose we have a pair $s,t$ of states of $Q_A$ which distribute similarly over $Q_A$, then all pairs $\phi^{i}$ and $t\phi^{i}$ also distribute similarly over $Q_A$ and, furthermore, the corresponding discriminant permutations must be trivial. We want to understand when it is possible to collapse the pairs $s\phi^{i}$, $t\phi^{i}$ in $A$ in order to obtain a smaller support of $H$. We consider first the case where $s$ and $t$ belong to the same orbit under $\phi$ and then the case where $s$ and $t$ belong to distinct orbits. The  focus is local initially: we consider pairs of states that distribute similarly along the orbit of a state. We will later  stitch the local pictures into a global one.

Lemmas~\ref{Lem:relabelling one orbit} and ~\ref{Lem:relabelling two orbit} which follow are illustrated in Example~\ref{Example:oneandtwoorbit}.

\begin{lemma}\label{Lem:relabelling one orbit}
    Let $A$ be a synchronising automaton and $\phi$ an automorphism of the underlying digraph of $A$. Fix $N \in \N$. Suppose the following conditions hold:
    \begin{itemize}
        \item $A$ has a single orbit of states;
        \item all edges of $A$ have the same orbit length $N$;
        \item there is an $r \in \N_{1}$ such that for any  $i \in \N$ and any $s \in Q_{A}$, the maps $\pi_{A}(\cdot, s\phi^{i})$ and $\pi_{A}(\cdot, s\phi^{i+r})$ coincide.
    \end{itemize}
    Then there is a relabelling $B$ of $A$ such that  
    \begin{itemize}
        \item the maps $\pi_{B}(\cdot, s\phi^{i})$ and $\pi_{B}(\cdot, s\phi^{i+r})$ coincide for any $i \in \N$ and any $s \in Q_A$;
        \item for any $x \in \xn$, the edges $(s\phi^{i},x,\pi_{B}(x, s\phi^i))\phi$ and $(s\phi^{i+r},x,\pi_{B}(x, s\phi^i))\phi$ have the same label for any $i \in  \N$;
        \item Let $B'$ be the folding of $B$ which induces the partition of $Q_B$ corresponding to the orbits of $Q_B$ under $\phi^{r}$. There is an induced action of $\phi$ on $B'$ such that edges of $B'$ have the same orbit length $N$.
    \end{itemize}
    
\end{lemma}
\begin{proof}
    Let $r$ be minimal such that the third hypothesis holds. There is an $m \in \N$ such that the orbit length of any state of $A$ is $mr$.

    Let $a \in \N$ be such that $N = amr$.

    By  the hypothesis, given an $x \in \xn$, and  $1 \le i <m$, the edges $(s,x,\pi_{A}(x,s))$, $(s\phi^{ir},x,\pi_{A}(x,s))$ have disjoint orbits.

    Fix a state $s_0 \in Q_A$. Let $s_{ir + j} = s_0\phi^{ir +j}$ for all $0 \le i <m$ and $0 \le j < r-1$. We define the relabelling map $\lambda_{A}$ inductively as follows.

    Let $x \in \xn$ be smallest such that $\lambda_{A}(x,s_0)$ is not defined.

    Let $i$ so that $s_i = \pi_{A}(x,s_0)$. For $0 \le j \le m-1$, set $x_{j}$ as the label of the edge $(s_0, x, s_i)\phi^{jr}$. Note that by the assumption on $r$, there is an edge $(s_0, x_j, s_i\phi^{jr})$.

    The relabelling map $\lambda_{A}$ is defined such that the following holds: for all $0 \le l \le mr-1$ and $0 \le j \le m-1$ the label of the edge $(s_0,x_j,s_{i}\phi^{jr})\phi^{l}$ is set as the  label of the edge $(s_{jr}, x_j,s_{i}\phi^{jr})\phi^{l}$.

    Let $B'$ be as in the statement of the lemma. Let $\iota: Q_B \to B'$ be defined such that  $s\iota = s\iota =[s\phi^{jr}]$. The map $\iota$ extends to an automata isomorphism that preserves edge labels: i.e $(s,x,t)\iota = ([s],x,[t])$. Let $\phi'$ be the automorphism of $G_{B'}$ be uniquely defined by the rule $([s],x,[t])\phi' = (s,x,t)\phi\iota$.

\end{proof}

\begin{lemma}\label{Lem:relabelling two orbit}
    Let $A$ be a synchronizing automaton and $\phi$ an automorphism of the underlying digraph $G_{A}$ of $A$. Let $s,t,p$ be states of $A$ such that:
    \begin{itemize}
    \item  $\lets{s}{p} \ne \emptyset$;
        \item  for $i,j \in \N$ and $x\in \xn$,  $\pi_{A}(x,s\phi^{i}) = p\phi^{j}$ if and only if $\pi_{A}(x,t\phi^{i}) = p\phi^{j}$;
        \item the orbits of $s$ and $t$ are distinct and have equal length $l$; and
        \item  any edge $(s,x, p\phi^{j})$ has the same orbit length as the corresponding edge $(t,x,p\phi^{j})$.
    \end{itemize}
     Then there is a relabelling $B$ of $A$ such that the induced automorphism $\varphi$ of $G_{B}$ satisfies: for any $i,j \in \N$, $\lets{s\phi^{i}}{p\phi^{j}} = \lets{t\phi^{i}}{p\phi^{j}}$, and for any $x \in \lets{s\phi^{i}}{p\phi^{j}}$, the labels of the edges $(s\phi^{i}, x,p\phi^{j})\varphi$ and $(t\phi^{i}, x,p\phi^{j})\varphi$  coincide.
\end{lemma}

\begin{proof}
    This is a more straight-forward relabelling operation than the previous case. We simply match the orbits of $t$ along $p$ with those of $s$ along $p$. 
    
    

    We define the relabelling map $\lambda_{A}$ inductively.
    
    First, for any pair $(x,u) \in \xn \times Q_{A}$ such that $(u,x, \pi(x,u))$ is not an edge from a state in the orbit of $t$ to a state in the orbit of $p$, set $\lambda_{A}(x,u) =x$.
    
    Second, let $x \in \xn$ be smallest such that $(t,x, p\phi^{i})$ is an edge for some $i$ where $\lambda_{A}(x,t)$ is not defined. Let $N$ be the length of the orbit of the edge $(t,x, p\phi^{i})$.  Let $x_j$ be the label of the edge $(t,x, p\phi^{i})\phi^{j}$ for $0 \le j <N$.
    Set $\lambda_{A}(x_j,t\phi^{j}) =y$ where $y$ is the label of the edge $(s,x, p\phi^{i})\phi^{j}$.
    
    This inductively defined relabelling map $\lambda_{A}$ is given by a vertex fixing automorphism and induces the required relabelling of $A$.
    
\end{proof}

\begin{rem}
Note that in both  Lemmas~\ref{Lem:relabelling two orbit} and \ref{Lem:relabelling one orbit}, $B$ and $A$ are isomorphic as automata and so we  write $(A, \varphi)$ for $(B, \varphi)$ with the relabelling understood.
\end{rem}

We have an illustrative example.

\begin{example}\label{Example:oneandtwoorbit}

We illustrate the proofs of Lemma~\ref{Lem:relabelling one orbit} and Lemma~\ref{Lem:relabelling two orbit}. The automaton and elements of $\hn{n}$ we use in these example, as will be seen, arise as part of our process for showing that $H$ is conjugate to a $6$-cycle (see Examples~\ref{Example:addingshadowstatestoA}). However, we felt it would be of benefit to the reader to illustrate the relabelling processes now.  The notation in what follows has been chosen to match the two examples referred to.

\textit{Lemma~\ref{Lem:relabelling one orbit}}] 

We consider the element $H'$ which is the automaton depicted to the left of Figure~\ref{fig:Exmoneround}. It acts by an automorphism $\varphi_{H'}$ on the automaton $\bar{A}$ to the right of Figure~\ref{fig:Exmoneround} (which, an observant reader might notice, is the domain and range automaton of $H'$) as follows. The action on the vertices is the three cycle $(a_1 \, a_3 \, a_5)$; the action on the edges is determined by the requirement that $H' = \mathscr{H}(\bar{A}, \varphi_{H'})$ (equality in this case as $\bar{A}$ is the domain graph of $A$). Notice that all edges are on orbits of length $6$.

\begin{figure}[htb]
\begin{center}
\scalebox{0.85}{
\begin{tikzpicture}[shorten >=0.5pt,node distance=3cm,on grid,auto] 
\begin{scope}[xshift=0cm]
    \node[state] (q_1) [xshift=4cm,yshift=0cm] {$a_1$};
   \node[state] (q_3) [xshift=4cm,yshift=-3cm] {$a_5$}; 
   \node[state] (q_5)[xshift=8cm,yshift=-3cm] {$a_3$}; 
    \path[->] 
    (q_1) edge [in=140,out=320]  node [xshift=-0.5cm,yshift=0.5cm,swap]  {$4|2$} node [xshift=0.1cm,yshift=-0.03cm,swap]  {$5|3$} (q_5)
          edge [in=95,out=85, loop] node [xshift=-0.4cm,yshift=-0.3cm,swap] {$0|4$} node [xshift=0.4cm,yshift=-0.3cm,swap] {$1|5$} ()
          edge [out=260,in=100]  node [xshift=0cm,yshift=0.3cm,swap] {$2|1$} node [xshift=0cm,yshift=-0.3cm,swap]  {$3|0$} (q_3)
    (q_3) edge  node [xshift=0.8cm, yshift=0.3cm] {$1|4$}  node [xshift=0.8cm, yshift=-0.3cm] {$0|5$} (q_1)
          edge node [xshift=0.15cm, yshift=-0.1cm]  {$4|2$ \quad $5|3$} (q_5) 
          edge [in=275,out=265, loop] node [xshift=-0.3cm,yshift=0.2cm,swap] {$2|1$} node [xshift=0.3cm, yshift=0.2cm,swap] {$3|0$} ()
    (q_5) edge [in=330,out=135]  node [xshift=-0.2cm,yshift=0cm,swap] {$1|5$} node [xshift=0.2cm, yshift=-0.4cm,swap] {$0|4$} (q_1)
          edge [out=187,in=353] node [xshift=0.7cm,yshift=-0.6cm,swap] {$2|0$} node [xshift=-0.4cm,yshift=-0.6cm,swap] {$3|1$} (q_3) 
          edge [in=355,out=5, loop] node [xshift=0.3cm, yshift=0.4cm,swap] {$4|2$} node [xshift=0.3cm, yshift=-0.4cm,swap] {$5|3$} ();
    
\end{scope}
\begin{scope}[xshift=8cm]
    \node[state] (q_1) [xshift=4cm,yshift=0cm] {$a_1$};
   \node[state] (q_3) [xshift=4cm,yshift=-3cm] {$a_5$}; 
   \node[state] (q_5)[xshift=8cm,yshift=-3cm] {$a_3$}; 
    \path[->] 
    (q_1) edge [in=140,out=320]  node [xshift=-0.5cm,yshift=0.5cm,swap]  {$4$} node [xshift=0.1cm,yshift=-0.03cm,swap]  {$5$} (q_5)
          edge [in=95,out=85, loop] node [xshift=-0.4cm,yshift=-0.3cm,swap] {$0$} node [xshift=0.4cm,yshift=-0.3cm,swap] {$1$} ()
          edge [out=260,in=100]  node [xshift=0cm,yshift=0.3cm,swap] {$2$} node [xshift=0cm,yshift=-0.3cm,swap]  {$3$} (q_3)
    (q_3) edge  node [xshift=0.5cm, yshift=0.3cm] {$1$}  node [xshift=0.5cm, yshift=-0.3cm] {$0$} (q_1)
          edge node [xshift=0.15cm, yshift=-0.1cm]  {$4$ \quad $5$} (q_5) 
          edge [in=275,out=265, loop] node [xshift=-0.3cm,yshift=0.2cm,swap] {$2$} node [xshift=0.3cm, yshift=0.2cm,swap] {$3$} ()
    (q_5) edge [in=330,out=135]  node [xshift=-0.2cm,yshift=0cm,swap] {$1$} node [xshift=0.2cm, yshift=-0.4cm,swap] {$0$} (q_1)
          edge [out=187,in=353] node [xshift=0.5cm,yshift=-0.5cm,swap] {$2$} node [xshift=-0.2cm,yshift=-0.5cm,swap] {$3$} (q_3) 
          edge [in=355,out=5, loop] node [xshift=0cm, yshift=0.4cm,swap] {$4$} node [xshift=0cm, yshift=-0.4cm,swap] {$5$} ();
    
\end{scope}
\end{tikzpicture}
}
\end{center}
\caption{The element $H' \in \hn{6}$ and the automaton $\bar{A}$.}\label{fig:Exmoneround}
\end{figure}%

We follow the proof of Lemma~\ref{Lem:relabelling one orbit}. For this example, $m=3$ and $r=1$.

Set $s_0 = a_1$. In this first step, we take $x = 0$. We only need one step of the relabelling process.

The relabelling map $\lambda_{A}$ is expressed by the following matrix equality:
$$
\begin{pNiceMatrix}[first-row]
a_1 & a_3 & a_5 \\
(0,1) & (4,5) & (2,3) \\
(2,3) & (0,1) & (4,5) \\
(4,5) & (2,3) & (0,1) 

\end{pNiceMatrix} = \begin{pNiceMatrix}
(0,1) & (4,5) & (2,3) \\
(2,3) & (0,1) & (4,5) \\
(4,5) & (2,3) & (1,0) 

\end{pNiceMatrix}
$$
which is read as follows: for an entry $(x,y)$ under column header $a_i$ in the matrix on the LHS, if $(u,v)$ is the corresponding entry in the matrix on the RHS, we set $\lambda_{\overline{A}}(x, a_i) = u$ and $\lambda_{\overline{A}}(y, a_i) = v$.

 The resulting transducer is the element $J \in \hn{6}$ shown in Figure~\ref{fig:Exmconjugatorround2}. 
\begin{figure}[htb]
\begin{center}
\scalebox{0.85}{
\begin{tikzpicture}[shorten >=0.5pt,node distance=3cm,on grid,auto] 
\node [xshift=9cm, yshift=-1.5cm] {$\mathlarger{=}$};
\begin{scope}[xshift=0cm]
    \node[state] (q_1) [xshift=4cm,yshift=0cm] {$a_1$};
   \node[state] (q_3) [xshift=4cm,yshift=-3cm] {$a_5$}; 
   \node[state] (q_5)[xshift=8cm,yshift=-3cm] {$a_3$}; 
    \path[->] 
    (q_1) edge [in=140,out=320]  node [xshift=-0.5cm,yshift=0.5cm,swap]  {$4|4$} node [xshift=0.1cm,yshift=-0.03cm,swap]  {$5|5$} (q_5)
          edge [in=95,out=85, loop] node [xshift=-0.4cm,yshift=-0.3cm,swap] {$0|0$} node [xshift=0.4cm,yshift=-0.3cm,swap] {$1|1$} ()
          edge [out=260,in=100]  node [xshift=0cm,yshift=0.3cm,swap] {$2|2$} node [xshift=0cm,yshift=-0.3cm,swap]  {$3|3$} (q_3)
    (q_3) edge  node [xshift=0.8cm, yshift=0.3cm] {$1|0$}  node [xshift=0.8cm, yshift=-0.3cm] {$0|1$} (q_1)
          edge node [xshift=0.15cm, yshift=-0.1cm]  {$4|4$ \quad $5|5$} (q_5) 
          edge [in=275,out=265, loop] node [xshift=-0.3cm,yshift=0.2cm,swap] {$2|2$} node [xshift=0.3cm, yshift=0.2cm,swap] {$3|3$} ()
    (q_5) edge [in=330,out=135]  node [xshift=-0.2cm,yshift=0cm,swap] {$1|1$} node [xshift=0.2cm, yshift=-0.4cm,swap] {$0|0$} (q_1)
          edge [out=187,in=353] node [xshift=0.7cm,yshift=-0.6cm,swap] {$2|2$} node [xshift=-0.4cm,yshift=-0.6cm,swap] {$3|3$} (q_3) 
          edge [in=355,out=5, loop] node [xshift=0.3cm, yshift=0.4cm,swap] {$4|4$} node [xshift=0.3cm, yshift=-0.4cm,swap] {$5|5$} ();
    
\end{scope} 
\begin{scope}[xshift=8cm]
    \node[state] (q_1) [xshift=4cm,yshift=0cm] {$a_1$};
   \node[state] (q_3) [xshift=4cm,yshift=-3cm] {$a_5$}; 
    \path[->] 
    (q_1) edge [in=95,out=85, loop] node [xshift=-0.4cm,yshift=-0.3cm,swap] {$0|0$} node [xshift=0.4cm,yshift=-0.3cm,swap] {$1|1$}
    node [xshift=-0.4cm,yshift=0.3cm,swap] {$4|4$} node [xshift=0.4cm,yshift=0.3cm,swap] {$5|5$}()
          edge [out=260,in=100]  node [xshift=0cm,yshift=0.3cm,swap] {$2|2$} node [xshift=0cm,yshift=-0.3cm,swap]  {$3|3$} (q_3)
    (q_3) edge  node [xshift=0.8cm, yshift=0.3cm] {$1|0$}  node [xshift=0.8cm, yshift=-0.3cm] {$0|1$}
       node [xshift=1.5cm, yshift=0.3cm] {$4|4$}  node [xshift=1.5cm, yshift=-0.3cm] {$5|5$} (q_1)
          edge [in=275,out=265, loop] node [xshift=-0.3cm,yshift=0.2cm,swap] {$2|2$} node [xshift=0.3cm, yshift=0.2cm,swap] {$3|3$} ();
    
\end{scope}

\end{tikzpicture}
}
\end{center}
\caption{The relabelling transducer $J$.}\label{fig:Exmconjugatorround2}
\end{figure}

\textit{Lemma~\ref{Lem:relabelling two orbit}}

For this example we consider the element $H$ of Figure~\ref{fig:ExampleACallup} acting as an automorphism $\psi_{H}$ on the automaton $A'$ depicted in Figure~\ref{fig:ExmaddingshaowstoB}. On the vertices of  $A'$, $\psi_{H}$ induces the permutation $$(p_0 \, p_1 \, p_2 \, p'_0 \, p'_1 \, p'_2) (q_0 \, q_1 \, q_2);$$ the action on edges is determined by the fact that $H$ must be the minimal representative of $\mathscr{H}(A', \psi_{H})$.
\begin{figure}[htb]
\begin{center}
\scalebox{0.7}{    
\begin{tikzpicture}[shorten >=0.5pt,node distance=3cm,on grid,auto,scale=1, transform shape] 
\begin{scope}[xshift=0cm]
    \node[state] (p_0) [xshift=-2cm,yshift=1.5cm]   {$p'_0$}; 
    \node[state] (q_0)   {$p_0$}; 
    \node[state] (q_1) [xshift=4cm,yshift=0cm] {$q_0$};
   \node[state] (q_3) [xshift=4cm,yshift=-3cm] {$q_2$}; 
   \node[state] (q_5)[xshift=8cm,yshift=-3cm] {$q_1$}; 
   \node[state] (q_4)[xshift=8cm,yshift=0cm] {$p_1$};
   \node[state] (p_4)[xshift=10cm,yshift=1.5cm] {$p'_1$};
   \node[state] (q_2)[xshift=8cm,yshift=-6cm] {$p_2$};
   \node[state] (p_2)[xshift=10cm,yshift=-7.5cm] {$p'_2$};
   \node (d) [xshift = 15cm]{};
    \path[->] 
    (q_0) edge node [swap] {$1$ \quad $0$} (q_1)
          edge [out=330,in=135]  node [xshift=-0.3cm,yshift=0.cm] {$2$} node  [xshift=0.3cm,yshift=-0.4cm] {$3$} (q_3)
          edge [out=90,in=90]  node {$4$} (q_4)
          edge [out=90,in=90]  node {$5$} (p_4)
    (p_0) edge node {$1$} node [xshift=1cm,yshift=-0.3cm] {$0$} (q_1)
          edge [out=300,in=180]  node [xshift=-0.3cm,yshift=0.2cm,swap] {$2$} node  [xshift=0.4cm,yshift=-0.1cm,swap] {$3$} (q_3)
          edge [out=90,in=90]  node {$4$} (q_4)
          edge [out=90,in=90]  node {$5$} (p_4)
    (q_1) edge [in=170,out=10]  node  {$4$} (q_4)
          edge [in=170,out=40]  node {$5$} (p_4)
          edge [in=95,out=85, loop] node [xshift=-0.4cm,yshift=-0.3cm,swap] {$0$} node [xshift=0.4cm,yshift=-0.3cm,swap] {$1$} ()
          edge  node [xshift=0cm,yshift=0.3cm,swap] {$2$} node [xshift=0cm,yshift=-0.3cm,swap]  {$3$} (q_3)
    (q_3) edge  node {$1$} (q_0)
          edge [out=175,in=305]  node [swap] {$0$} (p_0)
          edge node [xshift=0.15cm, yshift=-0.1cm]  {$4$ \quad $5$} (q_5) 
          edge [in=275,out=265, loop] node [xshift=-0.3cm,yshift=0.2cm,swap] {$2$} node [xshift=0.3cm, yshift=0.2cm,swap] {$3$} ()
    (q_5) edge [in=330,out=135]  node [xshift=-0.2cm,yshift=0cm,swap] {$1$} node [xshift=0.2cm, yshift=-0.4cm,swap] {$0$} (q_1)
          edge [out=260,in=100] node [swap] {$2$} (q_2)
          edge [out=315,in=100] node {$3$} (p_2)
          edge [in=355,out=5, loop] node [xshift=0.1cm, yshift=0.3cm,swap] {$4$} node [xshift=0.1cm, yshift=-0.3cm,swap] {$5$} ()
    (q_4) edge [in=0,out=0]  node {$2$} (q_2)
          edge [in=0,out=0]  node {$3$} (p_2)
          edge  node [xshift=0cm,yshift=-0.4cm] {$5$} node [xshift=0cm,yshift=0.4cm] {$4$} (q_5) 
          edge node [xshift=-0.5cm, yshift=0.05cm] {$1$} node [xshift=0.2cm, yshift=0.05cm] {$0$} (q_1)
    (p_4) edge [in=0,out=0]  node  {$2$} (q_2)
          edge [in=0,out=0]  node {$3$} (p_2)
          edge  node [xshift=0cm,yshift=-0.4cm] {$5$} node [xshift=0cm,yshift=0.4cm] {$4$} (q_5) 
          edge [out=175,in=35] node [xshift=-0.5cm, yshift=0.05cm] {$1$} node [xshift=0.2cm, yshift=0.05cm] {$0$} (q_1)
    (q_2) edge [in=270,out=180]  node [swap] {$1$} (q_0)
          edge [in=225,out=225]  node [swap] {$0$} (p_0)
          edge  node [xshift=-0.4cm,yshift=0.2cm,swap] {$3$} node [xshift=0.4cm,yshift=-0.5cm,swap] {$2$} (q_3) 
          edge node [xshift=0cm, yshift=-0.4cm,swap] {$4$} node [xshift=0cm, yshift=0.4cm,swap] {$5$} (q_5)
    (p_2) edge [in=270,out=180]  node [swap] {$1$} (q_0)
          edge [in=225,out=225]  node {$0$} (p_0)
          edge [in=315,out=170]  node [xshift=-0.4cm,yshift=0.55cm] {$3$} node [xshift=0.4cm,yshift=-0.1cm] {$2$} (q_3) 
          edge [out=105,in=310] node [xshift=0cm, yshift=-0.2cm] {$4$} node [xshift=0cm, yshift=0.4cm] {$5$} (q_5);
\end{scope}
\end{tikzpicture}
}
\end{center}
\caption{The automaton $A'$.}\label{fig:ExmaddingshaowstoB}
\end{figure}%
The states $p_0$ and $q_0$ satisfy the hypothesis of Lemma~\ref{Lem:relabelling two orbit} for any choice of state upon which both are incident. We may consequently apply the relabelling strategy set out in the proof on all edges of $p_0$ and $q_0$ simultaneously. Working through similarly labelled edges of $p_0$ and $q_0$, beginning with the smallest, the strategy  relabels such that corresponding labels match along orbits. This results in a relabelling by the transducer $I$ given in Figure~\ref{fig:exmconjugator}. We note that $I$ is in fact the minimal representative of the transducer arising from the relabelling map, and its state set corresponds to the following partition of $Q_{A'}$: $$c_0:=\{p'_0\},c_1:=\{q_0,p_0\}, c_2:=\{p_2\},c_3:=\{q_2,p'_2\}, c_4:=\{p_1\}, c_5:=\{q_1,p'_1\}.$$ %
\begin{figure}[htb]
\begin{center}
\scalebox{0.8}{
\begin{tikzpicture}[shorten >=0.5pt,node distance=3cm,on grid,auto,] 
    \node[state] (q_0)   {$c_0$}; 
    \node[state] (q_1) [xshift=3cm,yshift=3cm] {$c_1$};
   \node[state] (q_3) [xshift=3cm,yshift=-3cm] {$c_3$}; 
   \node[state] (q_5)[xshift=6cm,yshift=0cm] {$c_5$}; 
   \node[state] (q_4)[xshift=6cm,yshift=6cm] {$c_4$};
   \node[state] (q_2)[xshift=9cm,yshift=3cm] {$c_2$};
    \path[->] 
    (q_0) edge node [xshift=-0.3cm,yshift=-0.5cm] {$1|0$} node [xshift=0.3cm,yshift=-0.1cm] {$0|1$} (q_1)
          edge [out=325,in=130]  node [xshift=-0.3cm,yshift=0.cm] {$2|3$} node  [xshift=0.3cm,yshift=-0.5cm] {$3|2$} (q_3)
          edge [out=90,in=180]  node  {$4|4$} (q_4)
          edge [out=40,in=140]  node [xshift=-0.6cm, yshift=-0.2cm]  {$5|5$} (q_5)
    (q_1) edge [in=230,out=40]  node [xshift=0cm,yshift=0.4cm,swap]  {$4|4$} (q_4)
          edge [in=95,out=85, loop] node [xshift=-0.4cm,yshift=-0.3cm,swap] {$0|0$} node [xshift=0.4cm,yshift=-0.3cm,swap] {$1|1$} ()
          edge  node [xshift=0cm,yshift=0.3cm,swap] {$2|2$} node [xshift=0cm,yshift=-0.3cm,swap]  {$3|3$} (q_3)
          edge  node [xshift=-0.3cm,yshift=0.7cm, swap]  {$5|5$} (q_5)
    (q_3) edge  node {$0|0$} (q_0)
          edge node [xshift=-0.25cm, yshift=-0.6cm]  {$4|4$} node [xshift=0.4cm, yshift=-0cm] {$5|5$} (q_5) 
          edge [in=275,out=265, loop] node [xshift=-0.3cm,yshift=0.2cm,swap] {$2|2$} node [xshift=0.3cm, yshift=0.2cm,swap] {$3|3$} ()
          edge [out=85,in=275,swap]  node [xshift=-0.1cm] {$1|1$} (q_1)
    (q_5) edge [in=325,out=130]  node [xshift=-1cm,yshift=0.5cm,swap] {$1|1$} node [xshift=0.3cm, yshift=-0.8cm,swap] {$0|0$} (q_1)
          edge [out=40,in=230] node [swap] {$2|2$} (q_2) 
          edge [in=275,out=265, loop] node [xshift=0.3cm, yshift=0.1cm,swap] {$4|4$} node [xshift=-0.3cm, yshift=0.1cm,swap] {$5|5$} ()
          edge [out=230,in=40] node [xshift=-0.2cm] {$3|3$} (q_3) 
    (q_4) edge node {$2|2$} (q_2)
          edge node [xshift=-0.2cm,yshift=1.8cm] {$3|3$} (q_3)
          edge  node [xshift=0cm,yshift=0.2cm] {$5|4$} node [xshift=0cm,yshift=0.8cm] {$4|5$} (q_5) 
          edge node [xshift=-0.7cm, yshift=0.6cm] {$1|0$} node [xshift=0cm, yshift=1.1cm] {$0|1$} (q_1)
    (q_2) edge  node[xshift=0.4cm, yshift=0.8cm] {$0|0$} (q_0)
          edge [out=140,in=40]  node[xshift=0.6cm, yshift=0.5cm] {$1|1$} (q_1)
          edge [out=270,in=0]  node [xshift=-0.2cm,yshift=0.1cm] {$3|2$} node [xshift=0.4cm,yshift=0.6cm] {$2|3$} (q_3) 
          edge node [xshift=0cm, yshift=-0.4cm,swap] {$4|5$} node [xshift=0.4cm, yshift=0cm,swap] {$5|4$} (q_5);
\end{tikzpicture}
}
\end{center}
\caption{The relabelling transducer $I$.}\label{fig:exmconjugator}
\end{figure}%
$\bigcirc$
\end{example}

 \subsection{More complex obstructions to a collapse: revealing shadow states} \label{Sec:Shadowstates}

In this subsection we address a more subtle obstruction to collapsing pairs of states along their orbits. Let $H \in \hn{n}$ be an element of finite order which admits an unimpeded support $A$ realised by an automorphism $\phi$ of its underlying digraph. Suppose $s$ and $t$ are states of $Q_A$ which distribute similarly over $Q_A$. It may be that $s$ and $t$ satisfy the first hypothesis of Lemma~\ref{Lem:relabelling two orbit} but there is no bijection from the edges of $s$ to the edges of $t$ that preserves both terminal vertex and orbit size. (This situation occurs in  Example~\ref{Example:minrepfromdual} for the states $p_0$ and $q_0$ of $\Duak{H}{2}$: $p_0$ and $q_0$ distribute similarly over $Q_A$, however, all edges of $q_0$ have orbit length $6$ but the edges $(p_0, 0, q_0)$ and $(p_0, 1,q_0)$ both have orbit length $3$.)  What we demonstrate in this section is that under suitable hypothesis one can add additional states, \textit{shadow states}, to $A$ such that $s$ and $t$ satisfy the hypothesis of either Lemma~\ref{Lem:relabelling one orbit} or Lemma~\ref{Lem:relabelling two orbit}. This then allows one of those two relabellings to succeed, and, in such a way, as we  will see, that enables a collapse down to a support of a conjugate of $A$ which has fewer states than $A$.

We begin now to restrict to the context of a freely acting finite order element of $\hn{n}$ all of whose orbits are of size $N$ for some divisor $N$ of $n$. 
\begin{quote}
\textbf{Throughout this section we fix a divisor $1< N$ of $n$}. 
\end{quote}

 \subsubsection{Revealing hidden dynamics}

The key tool, Lemma~\ref{lemma:blowingup}, developed in this section hinges on the following observation.

Let $H \in \hn{n}$ have finite order and  let $A$ be a support of $H$ realised by an automorphism $\phi_H$. 
Let $q \in Q_A$. 
Suppose  there is a minimal length $k$ so that all paths of length $k$ that end on $q$ have orbits of length $N$ (for some divisor of $N$ of $n$) under the action of $\phi_H$, and for this choice of $q$ we have $k > 1$.  
Let $\mathcal{P} = e_0e_1\ldots e_{k-1}$ be a path of length $k$ terminating at $q$, where the orbit of $\mathcal{P}$ has length $N$ but the orbit of $e_1e_2\ldots e_{k-1}$ is of size $m$ for some $m<N$. For indices $0\leq i\leq j\leq k-1$ set $\mathcal{P}_{i,j}\seteq e_ie_{i+1}\ldots e_j$.  By construction, the least common multiple of the orbit length of the edge $e_0$ and of $m$ is $N$, and further, the orbit length of $\mathcal{P}_{1,k-2}$ must divide $m<N$.  

    In the definition that follows, the state $t$ corresponds to the target of $e_0$ from the path $\mathscr{p}$ mentioned above, while $b$ is some integer multiple of the orbit length of $t$, but which still properly divides $N$.

\begin{predef}\label{def:heavystate}
    Let $H \in \hn{n}$, $A$ be a support of $H$ realised by an automorphism $\phi_{H}$ and $t \in Q_{A}$ a state.  We say \emph{$t$ is heavy (for the triple $(A,\phi_H,N)$)} if the following conditions hold:
    \begin{itemize}
        \item there is a proper divisor $b$ of $N$, where $b$ is divisible by the length of the orbit of $t$;
                \item for any $x \in \xn$ and any $s \in Q_{A}$ such that $(s,x,t)$ is an edge of $A$, the lowest common multiple of $b$ and  the length of the orbit of $(s,x,t)$  under $\phi_{H}$ is $N$. 
    \end{itemize}
 In this case, we call the value $b$ above a \emph{divisibility constant for $t$} and observe that the set of valid divisibility constants for $t$ might have more than one element.
\end{predef}

 Lemma~\ref{lemma:blowingup} characterises how to perform an \emph{in-split along the orbit of a heavy state $t$}, where, by ``in-split'' we mean the like-named operation for edge-shift equivalences (see, e.g. \cite{kitchens}).   The new automaton that is created has all of the old states, together with new states which we call \emph{shadow states (from the orbit of $t$)}. Targeted  applications of this lemma will enable us to resolve the obstruction discussed in the introduction to Section~\ref{Sec:Shadowstates}.

We begin with the following preliminary result.

\begin{lemma}\label{lem:blowupLengthsExist}
    Let $H \in \hn{n}$ and  let $A$ be a support of $H$ realised by an automorphism $\phi_{H}$.  Suppose there are $b\in \N$ and $t \in Q_{A}$ so that $t$ is heavy for the triple $(B,\phi_H,N)$ with $b$ a divisibility constant for $t$, and where $\big|\{t\phi_H^i|i \in\Z\}\big| =r$.
    
    In these circumstances, there is $M \in \N$ a divisor of the lengths of orbits of all edges $(s,x,t)$ where $M$ satisfies the following further conditions:
       
       \begin{enumerate}[label=\roman*)]
       \item \label{pt:lcm} the lowest common multiple of $M$ and $b$ is $N$,        
       \item \label{pt:bigFactor}there is $m>1$ so that $M=mr$.
       \end{enumerate}
\end{lemma}

\begin{proof}
Let $L$ be the greatest common divisor of the orbit lengths of all edges $(s,x, t)$ in $A$. Let $(s,x,t)$ be an edge of $A$ with orbit length $k$ under the action of $\langle \phi_H\rangle$.
Now, by the second bullet point of Definition~\ref{def:heavystate}, we see that $\lcm(k,b)=N$.
It follows, as $(s,x,t)$ is an arbitrary incoming edge for $t$, that $\lcm(L,b)=N$ as well.  Since $r$ is the orbit length of $t$ we see that $r|k$ and since $(s,x,t)$ is an arbitrary incoming edge for $t$  we therefore have $r|L$.  
By definition, $r|b$, so if $r=L$ we would have $\lcm(L,b)=b<N$ which is a contradiction.  It then follows that  $L = lr$ for some integer $l>1$.  Thus the set of numbers $M$ which divide the orbit lengths of all edges $(s,x,t)$ and  satisfy  points \ref{pt:lcm} and \ref{pt:bigFactor} is non-empty. Now let $M$ be an element of this set and determine $m\in\N$ so that $mr=M$ (noting that $1<m$ by construction).
\end{proof}

Any number $M$ which arises as in Lemma \ref{lem:blowupLengthsExist} below will be referred to as a \emph{valid splitting length for (the heavy state) $t$ (with respect to divisibility constant $b$).}

We now state and prove the key result of this section  which is also illustrated in Example~\ref{Example:addingshadowstatestoA} below.

\begin{lemma}[Shadow states lemma]\label{lemma:blowingup}
    Let $H \in \hn{n}$ and  let $A$ be a support of $H$ realised by an automorphism $\phi_{H}$.  Suppose there is $b\in \N$ and $t \in Q_{A}$ so that $t$ is heavy for the triple $(A,\phi_{H},N)$ with $b$ a divisibility constant for $t$.  Set $t_{0,0}= t$ and let $t_{0,0}, t_{0,1}\ldots, t_{0,r-1}$ be the orbit of $t$ under iteration by $\phi_{H}$. Let $M$ be a valid splitting length of $t$ and let $m>1$ be determined by $M = mr$.
       
   Under these hypotheses, we may form a new synchronizing automaton $A'$ with $$Q_{A'}= Q_{A} \sqcup \{t_{a,0}, \ldots, t_{a,r-1} \mid 1 \le a < m\}$$ such that 
       \begin{enumerate}[label=\roman*)]
            \item $\pi_{A'}(x,s)\seteq\pi_{A}(x,s)$ for those pairs $(x,s)\in \xn\times Q_{A}$ where $\pi_{A}(x,s)$ is not in the orbit of $t$;
           \item { $\pi_{A'}(\cdot, t_{a,i}) \seteq \pi_{A'}(\cdot, t_{0,i})$ for all $0 \le a < m$, $0 \le i < r$};
           \item the incoming transitions of $A'$ to the set of vertices $\{t_{a,i}\mid 0\leq a<m, 0\leq i<r \}$ are determined by the above rules, and by an automorphism $\psi_{H}$ of the underlying digraph $G_{A'}$ of $A'$ satisfying: 
           \begin{enumerate}
               \item $(t_{0,0})\psi_{H}^{ar + i} = t_{a,i}$ for $0 \le a < m$, $0\leq i< r$,
               \item  $(t_{0,0}) \psi_{H}^{M} = t_{0,0}$, and, \item $\mathscr{H}(A',\psi_{H}) = A$.
           \end{enumerate}
            
           \end{enumerate}
\end{lemma}

\begin{proof}
Let $M$ be a valid splitting length for the heavy state $t_{0,0}$ with divisibility constant $b$, and let $m>1$ be an integer so that $M = mr$.

Set
\[
T\seteq\{t_{0,0},t_{0,1},\ldots,t_{0,r-1}\} 
\]
and build a set of new objects (the  ``shadow states'' arising from the splitting along the orbit of $t_{0,0}$) 
\[
T'=\{t_{a,1},t_{a,2},\ldots,t_{a,r}\mid 1\leq a< m\}.
\] Note that $|T\cup T'|=M$.

We will define a bijection  $\psi_{H}$ of 
\[
Q_{A'}\seteq Q_A\cup T'
\]
as follows.

For $s\in Q_{A'}$, set 
\[
s\psi_{H}=
\left
\{\begin{matrix*}[l]
s\phi_{H}&\textrm{ if }&s\in Q_A\backslash T&\\
t_{a,i+1}&\textrm{ if }&s=t_{a,i}\textrm{ and }i<r-1&\\
t_{a+1,0}&\textrm{ if }&s=t_{a,r-1}\textrm{ and }a<m-1&\\
t_{0,0}&\textrm{ if }&s=t_{m-1,r-1}.&
\end{matrix*}
\right.
\]  
It is immediate by construction that $\psi_{H}$ is a bijection, and also that the orbit of $t_{0,0}$ has size $M$. 

Eventually we want $\psi_{H}$ and $A'$ to satisfy the conclusion of the lemma. At the moment we only have a state set for $A'$, alternatively, $G_{A'}$ has a vertex set $Q_A$ and no edges.
We will specify transitions for $A'$ by steadily  adding edges of the form $(p,x,q)$, for $p,q\in Q_{A'}$ and $x\in \xn$, to $G_{A'}$ (thus adding the transition $\pi_{A'}(x,p) = q$ to $A'$). We add these edges while simultaneously extending the function $\psi_{H}$ on the corresponding edges of $G_{A'}$ in such a way that the conclusions of the lemma are satisfied.  Important for ensuring this occurs  will be a graph homomorphism $\iota:G_{A'}\to G_{A}$.  On the set $Q_{A'}$, $\iota$ is defined as follows: for $s\in Q_{A'}\backslash (T\cup T')$ set $s\iota\seteq s$, and for any $t_{a,i}\in T\cup T'$ set $t_{a,i}\iota\seteq t_{0,i}$. Whenever we extend $G_{A'}$ by adding new edges, we also extend the graph homomorphisms $\iota:G_{A'}\to G_A$ and $\psi_{H}:G_{A'}\to G_{A'}$ so as to maintain \emph{rsc}, the \emph{rule of semi-conjugacy}, which we define here.
{\flushleft{\it \underline{rsc}}}:
\begin{enumerate}
\item \label{pt:nscState} for all $q\in Q_{A'}$ we have $q\psi_{H}\iota=q\iota\phi_{H}$, and
\item \label{pt:nscEdge} for all edges $e$ of $G_{A'}$  we further require $e\psi_{H}\iota=e\iota\phi_{H}$.
\end{enumerate}

Of course we have part (\ref{pt:nscState}) of the rule because we have already defined $\iota$ and $\psi_{H}$ over $Q_{A'}$ to satisfy this rule.  

In the above construction of $\iota$, if $e\iota=(s,x,t)$ then we will identify $e$ as $(p,x,q)$ where $p$ is the source of $e$ and $q$ is the target of $e$, so after any extension we can always think of the new $G_{A'}$  as an edge-labelled directed graph with edge labels ``lifted'' from $G_A$ by the map $\iota$.

Note that below we will sometimes add a large collection of edges at one go, but in this case, there is always a well defined triple $(p,x,q)$ for each new edge. This is because we add in edges along an orbit under $\psi_{H}$ which always contains a well-defined edge $(s,y,t)$, from which we can detect the correct letter labelling of all edges along the orbit by using rsc.

 It follows that if $A'$ is a synchronizing automaton then $\mathscr{H}(A,\phi_{H})$ will represent the same element of $\hn{n}$ as $\mathscr{H}(A',\psi_{H})$, since the map $\iota$ never changes edge labels, and the map $\psi_{H}$ will have to change edge labels in the corresponding fashion as $\phi_{H}$ in order to uphold rsc.

    We now begin to specify the edges of $G_{A'}$, and hence the transition function $\pi_{A'}$.  Recall below that $Q_A\backslash T = Q_{A'}\backslash (T\cup T')$.
    
    Partition the edges of $G_A$ into the following four sets.  
    \begin{align*}
    N_T\seteq&\{(p,x,q)\mid p,q\not\in T\},\\
     B_{T}\seteq &\{(p,x,q)\mid p,q\in T\},\\
 D_T\seteq &\{(p,x,q)\mid p\in T, q\not\in T\}, \textrm{ and}\\
 R_T\seteq &\{(p,x,q)\mid p\not\in T, q\in T\}.
    \end{align*}

We observe that each of the sets above is invariant under the action of $\phi_{H}$.

    For $(p,x,q)$ an edge in $N_T$, let $(p,x,q)$ also be an edge of $G_{A'}$ (and so $(x,p)\pi_{A'} = q$ as well) and set $(p,x,q)\psi_{H} \seteq (p,x,q)\phi_{H}$. 
    
    \newcommand{\tb}{\mathscr{Y}_B}
    \newcommand{\td}{\mathscr{Y}_D}
    \newcommand{\tr}{\mathscr{Y}_R}

    Let $\tb$ be a traversal for the orbits of the edges in $B_T$ such that each edge in $\tb$ is of the form $(t_{0,0},x,t_{0,i})$.  Similarly set $\td$ to be a traversal for the orbits of the edges in $D_T$ so that each edge of $\td$ is of the form $(t_{0,0},x,s)$ for some $s\in Q_A\backslash T$.  Finally set $\tr$ to be a traversal for the orbits of the edges in $R_T$ so that each edge of $\tr$ is of the form $(s,x,t_{0,0})$ for some $s\in Q\backslash T$.
    
    Extend $G_{A'}$ to include $\td\cup\tr\cup \tb$ as edges incident on $t_{0,0}$ (we will add more edges incident on $t_{0,0}$ later).  Furthermore, use the action of $\psi_{H}$ on the set $Q_{A'}$ together with the map $\iota$ to uniquely determine new edges (of the form $(p,x,q)$ for $x\in \xn$) that must be added to $G_{A'}$ so that the resulting digraph is closed under the action of $\psi_{H}$, contains the transversal edges $\td\cup\tr\cup\tb$ and satisfies rsc.  Note that this process extends the definition of $\iota$ and $\psi_{H}$ to these new edges as well, but these extensions are inductively well defined. Now we may use the new edges of $G_{A'}$ in the obvious way to also extend the definition of the transition function $\pi_{A'}$ so as to create a correspondingly larger automaton $A'$.

    Observe that for any state $s\in Q_{A'}\backslash(T\cup T')=Q_A\backslash T$, the process above now has created a unique edge of the
    form $(s,x,q)$ for each $x\in \xn$ (which are the ``lifts'' of $N_T$ and $R_T$ to $G_{A'}$ by $\iota$).  For edges in $N_T$ this is simply by definition.  For an edge $(s,x,t_{0,i})\in R_T$, there is an edge $(s',x',t_{0,0})\in \tr$ so that there is a minimal non-negative integer $k$ with $(s',x',t_{0,0})\phi_{H}^k=(s,x,t_{0,i})$.  It follows that $(s',x',t_{0,0})\psi_{H}^k=(s,x,t_{a,i})$ for the unique non-negative $a$ so that $k=ar+i$.  Now suppose there is an edge $(s,x,t_{c,j})$ of $G_{A'}$.  By rsc, we see that $(s,x,t_{c,j})\iota = (s,x,t_{0,j})$ but as there is a unique outgoing edge in $G_A$ from $s$ with letter $x$ we see that $t_{0,j}= t_{0,i}$ and in particular, $i= j$.  We assume without meaningful loss of generality that $c\geq a$ and that $|c-a|$ is minimal amongst all such differences.  Thus by rsc, the orbit length of the edge $(s,x,t_{0,i})$ under $\phi_{H}$ is precisely $(c-a)r$ or else $c=a$.  However, $(c-a)r< M = mr$ and, by definition, $M$ divides the length of the orbit of $(s,x,t_{0,i})$.  It follows that $(s,x,t_{a,i})$ is the unique pre-image of $(s,x,t_{0,i})$ under $\iota$.

There remains a special concern that we must address.  Specifically, there are now pairs $(x,t_{a,i})\in \xn\times (T\cup T')$ so that there are no edges of the form $(t_{a,i},x,q)$ in $G_{A'}$.  This happens as
the orbit of $t_{0,0}$ has length $r$ under $\phi_{H}$ but length $M=mr>r$ under $\psi_{H}$. Also, to verify the coherence of the rsc condition for edges in $\mathscr{Y}_{B}$, recall that $M$ divides the orbit length of these edges as they are in the orbit of an edge incident to $t$.

We now deal with this concern.  Observe that for an edge $(t_{0,0},x,s)\in \td\cup\tb$, its orbit under $\phi_{H}$ may contain multiple edges of the form $(t_{0,0},y,q)$ (for various $y\in\xn$ and $q\in Q_A$).  Let us organise these as the sequence of pairwise distinct edges $(e_0,e_1,\ldots,e_k)$ where $e_i = (t_{0,0},x,s)\phi_{H}^{ri}$, and with $e_k\phi_{H}^{r} = e_0=(t_{0,0},x,s)$.  In this context, the orbit of $(t_{0,0},x,s)$ under $\phi_{H}$ has length $(k+1)r$.  Let us set notation $e_i=:(t_{0,0},x_i,q_i)$ so we can understand the letter $x_i$ associated to $e_i$ for each valid index $i$.  The concern is that in our current graph $G_{A'}$, for indices $0<i\leq k$ with $i \not\equiv 0 \mod{m}$, there is no edge of the form $(t_{0,0},x_i,q)\in G_{A'}$. For the letter $x_i$ observe that there is an edge of the form $(t_{a,0},x_i,\bar{q}_i)$ of $G_{A'}$ for $a = ir \mod M$ and some state $\bar{q}_i$ of $G_A$ (note that if $q_i \notin T$, then  $\bar{q}_i= q_i$). The rule of modification is, add the edge $(t_{0,0},x_i,\bar{q}_i)$ to $G_{A'}$, for all indices $0<i\leq k$ with $i \not\equiv 0 \mod{m}$.  Repeat this same procedure across all of the transversal elements $(t_{0,0},x',s')\in \td\cup\tb$ and as a consequence, for each letter $y\in \xn$, we see that the vertex $t_{0,0}$ now has a unique outgoing edge of the form $(t_{0,0},y,q)$. Finally, we again use the action of $\psi_{H}$ on vertices and the action of $\phi_{H}$ on $G_{A}$ along with the rsc condition to extend the definitions of $\psi_{H}$ and $\iota$ to the necessary edges we have to add to $G_{A'}$ in order to complete the orbits of our newly-added edges based at $t_{0,0}$, and to discern what letters needed to be associated to these new edges.  Now induce from $G_{A'}$ the enlarged automaton $A'$.

One observes that for any valid indices $a$ and $b$ and fixed index $i$, the states $t_{a,i}$ and $t_{b,i}$ of $A'$ have all the same outgoing transitions, and indeed, that the automaton $A'$ collapses back down to $A$ by identifying these states for each fixed $i$.  In particular $A'$ is synchronizing as it admits a collapse sequence to the $n$-leafed rose.  Further, the rsc condition implies that $\psi_{H}$ acts as an automorphism of the directed graph $G_{A'}$ in fashion locally emulating how $\phi_{H}$ acts on $G_A$ so that $\mathscr{H}(G_{A'},\psi_{H})$ represents $A$.

\end{proof}

Let $H, A, t=t_{0,0}$ be as in the statement of Lemma~\ref{lemma:blowingup}.  Assume we applied Lemma~\ref{lemma:blowingup} to lengthen the orbit of $t$ as in the lemma statement to create an automaton $A'$ with automorphism $\psi_{H}$ so that $\mathscr{H}(A',\psi_{H})$ has minimal representative $A$ and where the orbit of $t$ in $G_{A'}$ under the action of $\psi_{H}$ is the set $\{t_{a,i}\mid 0\leq a<m, 0\leq i<r\}$.  Now for each state $t_{0,i}$ for $0\leq i<r$ (these states are in the original orbit of $t$ in $G_{A}$ under the action of $\phi_{H}$),  we call the set of states $$\{t_{a,i} \mid 0 < a \le m-1\}\subsetneq Q_{A'}$$ the
\textit{shadow states for $t_{0,i}$ (in $Q_{A'}$)}.   Note that these are  states of $\mathscr{H}(A',\psi_{H})$ with local maps equivalent to the local map at $t_{0,i}$ for the transducer $\mathscr{H}(A,\phi_{H})$.  If we apply Lemma~\ref{lemma:blowingup} inductively and perhaps repeatedly on states on the now extended orbit of $t$, we extend the definition of the shadow states of $t_{0,i}$ to be the union of the sets of states added in each round of applying Lemma~\ref{lemma:blowingup} which have local maps  equivalent to the local map at $t_{0,i}$ for the transducer $\mathscr{H}(A,\phi_{H})$.  Note that  this process cannot go on forever as each application of Lemma~\ref{lemma:blowingup} lengthens the orbit of $t_{0,0}$ with $N$ an upper bound on the  length of this orbit.  Also note that each added state will be a shadow state of one of the original states $t_{0,j}$ after any number of iterated applications of Lemma~\ref{lemma:blowingup}.

We illustrate the proof of Lemma~\ref{lemma:blowingup} using the element $H \in \hn{6}$ in Example~\ref{Example:minrepfromdual}.

\begin{example}\label{Example:addingshadowstatestoA}

We return to the context and notation  of Example~\ref{Example:minrepfromdual}. Recall that 
$H \in \hn{6}$ is the element depicted in Figure~\ref{fig:ExampleACallup}. 

The state $p_0$ is a heavy state for the triple $(A:=\Duak{A}{2}, \phi_{H}, 6)$ with divisibility constant $3$ equal to the orbit length of $p_0$. There is only one valid splitting length for $p_0$: we must have $M= 6$ since the orbit length of $p_0$ is $3$ and all incoming edges into $p_0$ have orbit length $6$.

The new automaton $A'$ constructed as in Lemma~\ref{lemma:blowingup} is depicted in Figure \ref{fig:ExmaddingshaowstoB}.%

The states  $p'_0 := {p_{0}}_{_{1,0}}, p'_1:={p_{0}}_{_{1,1}}$
and $p'_2:= {p_{0}}_{_{1,2}}$ are the shadow states of $p_0$. Let $\psi_{H}$ be a lift  of $\phi_{H}$ to $G_{A'}$ as described in the proof of Lemma~\ref{lemma:blowingup}. The action of $\psi_{H}$ is uniquely determined by the facts that the orbit of $p_0$ under $\psi_{H}$ is $(p_0 \; p_{1} \; p_{2} \; p'_0 \; p'_{1} \; p'_{2})$ and $\mathscr{H}(A', \psi_{H})$ has minimal representative $H$.
$\bigcirc$
\end{example}

Lemma~\ref{Lemma:alledgesonlengthn}  and Corollary~\ref{cor:allstatesonlengthn}, both consequences of Lemma~\ref{lemma:blowingup}, are the heart of the matter.

\begin{lemma}\label{Lemma:alledgesonlengthn}
    Let $H \in \hn{n}$ and  let $A$ be a support of $H$ realised by an automorphism $\phi_{H}$ of its underlying digraph. Suppose that all circuits in $A$ are on orbits of length $N$ under the action of $H$. Then there is a support $\widehat{A}$ realised by an automorphism $\widehat{\psi}_{H}$ of $G_{\widehat{A}}$ such that any edge of $\widehat{A}$ is on an orbit of length $N$ under the action of $H$.
\end{lemma}
\begin{proof}
We first observe that all states of $A$ must be on orbits of length dividing $N$ as the underlying digraph of $A$ is synchronizing and therefore each state is visited by some circuit which is on an orbit of length $N$.  Also, we may assume that $|A|>1$ otherwise all loops at the state of $A$ will be on orbits of length $N$ and we would be done.

Let $s \in Q_{A}$ be such that $s$  is on an orbit of length strictly less than $N$ under $\phi_{H}$.  (If all states of $Q_A$ were on orbits of length $N$ then all edges of $A$ would be on orbits of length $N$ as well and we would be done.)

Inductively define states as follows. 

Set $Q_{A}(0,s):= \{s\}$. Assume $Q_{A}(i,s)$ is defined for some $i \in \N$. An element $q \in Q_{A}$ belongs to $Q_{A}(i+1, s)$ if the following conditions hold: 
\begin{enumerate}
    \item \label{defQBpathconformance}there are elements $x_0, x_1, \ldots, x_{i} \in \xn$, such that, for all $0 \le j \le i$, we have $\pi_{A}(x_{i}\ldots x_{j-1}, q) \in Q_{A}(j-1, s)$; and
    \item \label{defQBOrbitLength}the path $(q, x_{i}x_{i-1} \ldots  x_{0}, s)$ is on an orbit of length strictly less than $N$ under $\phi_{H}$. 
\end{enumerate} 
We observe that for any state $q$ in a set $Q_A(i+1,s)$, the orbit of $q$ under $\phi_{H}$ has size properly dividing $N$.  If $q\in Q_{A}(i+1,s)$ and $x_{i},x_{i-1},\ldots,x_0\in\xn$ satisfies points (\ref{defQBpathconformance}) and (\ref{defQBOrbitLength}) then we call the path $(q,x_{i}x_{i-1}\ldots x_0,s)$ \emph{conformant for $Q_{A}(i+1,s)$}.

Let $k \in \N$ be minimal so that $Q_{A}(k+1,s) = \emptyset$.  If such $k$ did not exist then there would a long path (as in point (\ref{defQBOrbitLength}) of the definition of the sets $Q_A(i,s)$) which is long enough that it must contain a circuit in $A$.  Any such circuit would be on an orbit of length strictly less than $N$ under the action of $\phi_{H}$, which is a contradiction.

By the argument above it also follows that whenever $j > 0$, $s \notin Q_{A}(j,s)$.

We now apply an induction argument using Lemma~\ref{lemma:blowingup} to iteratively add shadow states to $A$ so as to reduce $k$ to $0$.

Let $t \in Q_{A}(k,s)$ and fix $x_{k-1}, \ldots, x_{0} \in  \xn$, such that the path $(t, x_{k-1} \cdots x_{0},s)$ is conformant for $Q_A(k,s)$ and where the orbit of this path under the action of $\phi_{H}$ is of length $b < N$ (note that $b|N$ and also that the length of the orbit of $t$ divides $b$). Moreover, by choice of $k$, for any pair $(x,p) \in \xn \times Q_{A}$ with $\pi_{A}(x,p) = t$, the lowest common multiple of the length of the orbit of the edge $(p,x,t)$ and $b$ is $N$. 

 It follows that $t$ is heavy for  $(A,\phi_{H},N)$ and $b$ is a divisibility constant for $t$, so we may apply Lemma~\ref{lemma:blowingup} with the state $t$ and constant $b$ to add states in the orbit of $t$ and necessary edges to form a new synchronizing automaton $A'$ supporting $A$ and realised by an automorphism $\psi_{H}$ of $G_{A'}$. In particular, all circuits of $G_{A'}$ are still on orbits of length $N$ under the action of $\psi_{H}$).

Recall that by the construction of $A'$, the orbit of $t$, which includes all of its shadow states, is now of larger size $N_t'$ under the action of the resulting digraph automorphism $\psi_{H}$. Thus, for any $t'$ in the orbit of $t$ (including the shadow states we have just added), if there is a path from $t'$ to $s$ which is conformant for $Q_{A'}(k,s)$, then $t'$ (and therefore $t$) is heavy for $(Q_{A'},\psi_{H},N)$, so we may again inductively increase the length of the orbit of $t$ by adding more shadow states until there are no paths from a point in the orbit of $t$ to $s$ which are conformant for $Q_{A'}(k,s)$ (note this happens to be a consequence of $t$ no longer being heavy for $(A',\psi_{H},N)$ which must happen eventually as the orbit of $t$ is getting longer and is bounded above by $N$). Note that if $p\in Q_{A'}(k,s)$ but $p$ is not in the orbit of $t$, then $p\in Q_{A}(k,s)$.  In particular, we have  $|Q_{A'}(k,s)|<|Q_{A}(k,s)|$ as this count of states for $A'$ no longer includes the state $t$ nor any state in its orbit.

We can now inductively repeat this process for $s$ until $|Q_{A'}(k,s)|=0$.

Note that we can repeat this process for any state $p$ with $|Q_{A'}(k,p)|\neq 0$.  Thus we may now proceed inductively in this fashion until finally we have constructed an automaton $A'$ and an automorphism $\psi_{H}$ of $G_{A'}$ so that $A'$ folds onto $A$ and $\mathscr{H}(A',\psi_{H})$ represents $A$, and where if $q$ is any state and $j$ is minimal so that $Q_{A'}(j,q)= \emptyset$, then $j=0$. 

We set $\widehat{A}=A'$ and $\widehat{\psi}_{H}=\psi_{H}$ in this final case, noting that the orbit of every edge of $\widehat{A}$ under the action of $\widehat{\psi}_{H}$ is of length $N$.
\end{proof}

\begin{rem}
Let $H \in \hn{n}$ be an element of finite order and suppose that every point in $\xnN$ is on an orbit of length $N$ under the action of $H$. By lemma~\ref{Lemma:alledgesonlengthn}, there is a minimal (in size) synchronizing automaton $A$ such that $H$ acts as an automorphism $\phi_{H}$ of the underlying digraph of $A$ and all edges of $A$ are on orbits of length $N$ under the action of $H$.
\end{rem}

\begin{example}\label{Example:allorbitsonlengthn}
    Consider once more the element $H \in \hn{6}$ given in Figure~\ref{fig:ExampleACallup}. Its minimal support is the automaton $A$ in Figure~\ref{fig:Exmdigraph}. As discussed in Example~\ref{Example:addingshadowstatestoA}, the edges $(p_0, s, q_0)$ ($x \in \{0,1\}$) are on orbits of length $3$. Taking $s$ in the proof Lemma~\ref{Lemma:alledgesonlengthn} to be the state $q_0$ we compute;
    $$
       Q_B(0,q_0) = \{s\}, \quad
       Q_B(1,q_0) = \{p_0,p_1\}, \quad
       Q_B(2, q_0)  = \emptyset.
  $$

    We only need one application of Lemma~\ref{lemma:blowingup}. Proceeding as in Example~\ref{Example:addingshadowstatestoA}, we obtain the support $A'$ of $H$ realised by an automorphism $\psi_{H}$ of $G_{A'}$ under which all edges are on orbits of length $6$. $\bigcirc$
\end{example}

\begin{cor}\label{cor:allstatesonlengthn}
  Let $H \in \hn{n}$ and  let $A$ be a support of $H$ realised by an automorphism $\phi_{H}$ of its underlying digraph. Suppose that all circuits in $A$ are on orbits of length $N$ under the action of $H$. Then there is a support $\widehat{A}$ of $H$ realised by an automorphism $\widehat{\psi}_{H}$ of $G_{\widehat{A}}$ such that all states of $\widehat{A}$ are on orbits of length $N$ under the action of $\psi_{H}$.  
\end{cor}
\begin{proof}
    By applying Lemma~\ref{Lemma:alledgesonlengthn}, we may assume that all edges in $A$ have orbit length $N$.

    Suppose there is a state $s \in Q_A$ which has orbit length strictly less than $N$. Then, notice that $s$ is a heavy state. We thus apply the shadow states lemma, with $M = N$ and $b=1$. This creates a new support $A'$ of $H$ in which $s$ now has orbit length $N$.

    We repeat the process until all states have  orbit length $N$.
\end{proof}

\subsection{Rebraiding operation}

Notice that Lemma~\ref{Lem:relabelling one orbit} has the strong hypothesis that $A$ has only one orbit under the action of $\phi_{H}$. If $A$ has more than one orbit of states and we can always find a pair $(s,t)$ of states on distinct orbits satisfying the hypothesis of Lemma~\ref{Lem:relabelling two orbit}, then by repeatedly following the process of applying that lemma and then merging parallel orbits of equivalent states, we (at the possible cost of passing to some conjugate element) eventually fall into the regime of Lemma~\ref{Lem:relabelling one orbit}. However, if there is more than one orbit of states, it is not obvious that it is always possible to find such a pair $(s,t)$ of states satisfying the Hypotheses of Lemma~\ref{Lem:relabelling two orbit}. Lemma~\ref{lemma:rebraiding} below allows us, under suitable hypothesis, to  alter the automaton, without changing the number of states and while preserving the conjugacy class of $H$, such that we may find a pair $(s,t)$ of states satisfying the hypothesis of Lemma~\ref{Lem:relabelling two orbit}. We will show in Section~\ref{sec:conclusion} that the hypotheses of  Lemma~\ref{Lem:relabelling two orbit} are satisfied whenever $A$ has more than one orbit of states and all states and edges have the same orbit length $N$. 

\begin{lemma}[Rebraiding]\label{lemma:rebraiding}
    Let $H \in \hn{n}$ be an element of finite order with $A$ an unimpeded support of $H$ realised by an automorphism $\phi_H$. Suppose all states and all edges of $A$ have orbit length $N$. Let $s,t$ be states of $H$ which belong to distinct orbits. Suppose for all $x \in \xnp$, $\pi_{A}(x,s)$ and $\pi_{A}(x,t)$ belong to the same orbit. Then there is a support $\widehat{A}$ of a conjugate $\widehat{H}$ of $H$ realised by an automorphism $\phi_{\widehat{H}}$, such that the following things hold:
    \begin{itemize}
        \item there is a bijection $\iota: Q_{A} \to Q_{\widehat{A}}$ such that $\iota\phi_{\widehat{H}} = \phi_{H}\iota$;
        \item for $x \in \xnp$, $\pi_{\widehat{A}}(x, s\iota) = \pi_{\widehat{A}}(x,t\iota)$.
    \end{itemize}
\end{lemma}

\begin{proof}

     Set $k \in \N$ maximal such that $\pi_{A}(w,s) \ne \pi_{A}(w,t)$ for some $w \in  \xn^{k}$.
     
     We inductively construct an automaton $A'$. Initially $Q_{A'} = Q_{A} \sqcup Q'$ where $Q' = \emptyset$ and $\pi_{A'} = \pi_{A}$.  Simultaneously, we build a map $\psi_{H}$  realising an automorphism for $A'$ as a support of $H$. Initially, $\psi_{H} = \phi_{H}$.

     For each $x \in \xn$ such that $\pi_{A}(x,s) \ne \pi_{A}(x,t)$, we consider two cases.

     If $p:= \pi_{A}(x,t) \ne t$, define new states $p_{x,0}, \ldots, p_{x,N-1}$. Consider the edge $(t,x,p)$. For each $0\leq i< N$ set the notation $(t\phi_{H}^{i},x\phi_{H}^{i},p\phi_{H}^{i}):=(t,x,p)\phi_{H}^{i}$. Now set $\pi_{A'}(x\phi_{H}^{i}, t\phi_{H}^{i}) = p_{x,i}$. For all $y \in \xn$, we set $\pi_{A'}(y, p_{x,i}) = \pi_{A}(y, p\phi_{H}^{i})$. Note that the states in the orbit of $p$ still has incoming edges from the orbit of $s$ while $N$ of the edges from the orbit of $t$ that used to target the orbit of $p$ now have targets amongst the states $p_{x,i}$.  The map $\psi_{H}$ acts on the new edges and new states in the obvious way due to the constraint that $\mathscr{H}(A', \psi_{H})$ represents $H$.

          If $\pi_{A}(x,t) = t$, then $\pi_{A}(x,s)\neq t$ is in the orbit of $t$ and so is not $s$ and in particular $\pi_{A}(x,s) \not\in \{s,t\}$. In this case we may follow  the previous paragraph setting $p:= \pi_{A}(x,s)$ and replacing  replacing $t$ by $s$ in the construction.  The map $\psi_{H}$ is extended as before.
          
     Set $Q' = Q'(0)$ to be all new states added above.
     
     For any other pair $(x,p)$ which is not covered in the cases above, set $\pi_{A'}(x,p) = \pi_{A}(x,p)$.

     Note that after this step
     \begin{itemize}
         \item $A'$ is still a synchronizing support of $H$;
         \item for all $xy \in \xn^{2}$, $\pi_{A'}(xy, s)$ and $\pi_{A'}(xy,t)$ belong to the same orbit in $A'$ under $\psi_{H}$;
         \item at most one of $\pi_{A'}(x, s)$ and $\pi_{A'}(x,t)$ is an element of $Q'(0)$; 
         \item $\pi_{A}(x, q) \not\in Q'(0)$ for any $x \in \xn$ and $q \in Q'(0)$.
     \end{itemize}

     For each $x_0x \in \xn^{2}$ such that $\pi_{A}(x_0x, s) \ne \pi_{A}(x_0x, t)$ we do the following. First note that exactly one of $\pi_{A'}(x_0,s)$ and $\pi_{A'}(x_0,t)$ is an element of  $Q'(0)$. We assume that $q':= \pi_{A'}(x_0, t) \in Q'(0)$ (the other case is handled similarly). Then for $p = \pi_{A'}(x_0x, t)$, we form new states $p_{x_0x,0}, p_{x_0x,1}, \ldots, p_{x_0x,N-1}$. Let $q = \pi_{A}(x_0, t)$ and note that $\pi_{A}(x, q) = p$. For all $0\leq i<N$ set the notation $(q\phi_{H}^{i},x\phi_{H}^{i}, p\phi_{H}^{i})\seteq (q,x,p)\phi_{H}^{i}$ and redefine $\psi_{H}$ and $A'$ using the rule $\pi_{A'}(x\phi_{H}^{i},q'\psi_{H}^{i}) =p_{x_0x, i}$ (for each $i$, this will remove an edge from $A'$ and $\psi_{H}$ while adding a new edge and state to $A'$). Finally, we set $\pi_{A'}(y,p_{x_0x,i}) = \pi_{A'}(y, p\psi_{H}^{i})$ for all $y \in \xn$. Extend $\psi_{H}$ in the natural way.

     Set $Q'(1)$ to be all new states added in the process just described. So that $Q' = Q'(0) \sqcup Q'(1)$. Note that after this step
     \begin{itemize}
         \item $A'$ is still a synchronizing support of $H$ realised by $\psi_{H}$;
         \item for all $x \in \xn^{3}$, $\pi_{A'}(x, s)$ and $\pi_{A'}(x,t)$ belong to the same orbit in $A'$ under $\psi_{H}$;
         \item  for all $x \in \xnp \cap \xn^{\le 2}$ at most one of $\pi_{A'}(x, s)$ and $\pi_{A'}(x,t)$ is an element of $Q'(|x|-1)$;
         \item $\pi_{A}(x, q) \not\in Q'(1)$ for any $x \in \xn$ and $q \in Q'(1)$.
     \end{itemize}

   Recall that $k$ is maximal such that for some $w \in \xn^{k}$, $\pi_{A}(w,s) \ne \pi_{A}(w,t)$. We carry forward the constructive process above until we have formed new states $Q'(k-1)$ so that $Q' = \bigsqcup_{i=0}^{k-1}Q'(i)$ such that the following holds:
    \begin{itemize}
        \item $A'$ is still a synchronizing support of $H$ realised by $\psi_{H}$;
        \item for all $x \in \xn^{k+1}$, $\pi_{A'}(x, s)$ and $\pi_{A'}(x,t)$ belong to the same orbit in $A'$ under $\psi_{H}$ (and are in fact the same state by the definition of $k$); 
        \item for all $x \in \xnp \cap \xn^{\le k}$ at most one of $\pi_{A'}(x, s)$ and $\pi_{A'}(x,t)$ is an element of $Q'(|x|-1)$;
        \item for all $x \in \xn$ and $q \in Q'(k-1)$ we have $\pi_{A}(x, q) \not\in Q'(k-1)$.
    \end{itemize}

    We may strengthen the third bullet point as follows: for any $x \in \xnp \cap \xn^{\le k}$ such that $\pi_{A}(x,s) \ne \pi_{A}(x,t)$, the orbits of the states $\pi_{A'}(x,s)$ and $\pi_{A'}(x,t)$ under $\psi_{H}$ are distinct in $A'$.

    Below, we successively relabel and collapse $A'$ to obtain the desired automaton $\widehat{A}$.

    Fix a word $w \in \xn^{k}$ such that $\pi_{A}(w,s) \ne \pi_{A}(w,t)$. Then $p:=\pi_{A'}(w,s)$ and $p':=\pi_{A'}(w,t)$ belong to distinct orbits of $A'$ under $\psi_{H}$ (exactly one belongs to $Q'(k-1)$). Assume that $p' \in Q'(k-1)$ (the other case is dealt with analogously).  Now as $\pi_{A'}(x,p) = \pi_{A'}(x,p')$, for all $x \in \xn$, since $A$ was assumed to be unimpeded, it follows that $\pi_{A'}(x, p\psi_{H}^{i}) = \pi_{A'}(x, p'\psi_{H}^{i})$ for all $x \in \xn$ and all $0 \le i \le N-1$. Thus, by Lemma~\ref{Lem:relabelling two orbit}, we may relabel, by a vertex fixing automorphism of $A'$, such that for any $x \in \xn$ and any $0 \le i \le N$ the edges $(p, x, \pi_{A'}(x,p))\psi_{H}^{i}$ and $(p',x, \pi_{A'}(x,p))\psi_{H}^{i}$ have identical labels. We may then identify each pair of states $p\psi_{H}^{i}$ and $p'\psi_{H}^{i}$ while preserving an action by a conjugate of $H$ on the resulting automaton. 

    We now repeat this process along all words $w \in \xn^{k}$ such that $\pi_{A}(w,s) \ne \pi_{A}(w,t)$. Let $A''$ be the resulting automaton $\psi_{H'}$ be the resulting induced automorphism. We observe that 
    \begin{itemize}
        \item for all $x \in \xn^{k}$, $\pi_{A'}(x, s)=\pi_{A'}(x,t)$ belong to the same orbit in $A''$ under $\psi_{H'}$; 
        \item for all $x \in \xnp \cap \xn^{\le k-1}$ at most one of $\pi_{A'}(x, s)$ and $\pi_{A'}(x,t)$ is an element of $Q'(|x|-1)$;
        \item $\pi_{A}(x, q) \not\in Q'(k-2)$ for any $x \in \xn$ and $q \in Q'(k-2)$.
    \end{itemize}

    We can thus repeat the process. Eventually, we have an automaton $\widehat{A}$ with, $Q_{\widehat{A}} = Q_{A}$, an induced automorphism $\phi_{\widehat{H}}$ of $G_{A}$ which acts in the same way as $\phi_{H}$ does on $Q_A$, and such that $\pi_{\widehat{A}}(x, s\phi_{\widehat{A}}^{i})=\pi_{A'}(x,t\phi_{\widehat{H}}^{i})$ for all $x \in \xn$ and $0 \le i \le N-1$.
     
\end{proof}

\begin{cor}\label{cor:rebraidingreduction}
    Let $H \in \hn{n}$ be an element of finite order. Let $A$ be an unimpeded support of $H$ realised by an automorphism $\phi_H$ and such that  all states and all edges of $A$ have orbit length $N$. Let $s$ and $t$ be states of $H$ which belong to distinct orbits. Suppose for all $x \in \xnp$, $\pi_{A}(x,s)$ and $\pi_{A}(x,t)$ belong to the same orbit. Then there is a support $\overline{A}$ of a conjugate $\overline{H}$ of $H$ realised by an automorphism $\phi_{\overline{H}}$, such that the following things hold:
    \begin{itemize}
        \item $|\overline{A}|< |A|$;
        \item all states and all edges of $\widehat{A}$ have orbit length $N$ under $\phi_{\overline{H}}$.
    \end{itemize}
\end{cor}
\begin{proof}

Let $(\widehat{A},\phi_{\widehat{A}})$ be the automaton obtained from $(A, \phi)$ as in the statement of Lemma~\ref{lemma:rebraiding}. We may identify, using Lemma~\ref{Lem:relabelling two orbit}, the states $s,t$ along their orbits to obtain a pair $(\overline{A}, \phi_{\overline{H}})$ supporting a conjugate $\overline{H}$ of $H$. Clearly, $|\overline{A}|< |A|$. Since the orbits of $s$ and $t$ are disjoint, then $(\overline{A}, \phi_{\overline{H}})$ still retains the property that all states and all edges have orbit length $N$.
    
\end{proof}

\section{Stitching it all together}\label{sec:conclusion}

We have now developed all the tools necessary to prove the main result. 

\begin{lemma}\label{lemma:allorbitsonlegnthniffallcircuitsonlength}
Let $H \in \hn{n}$ be an element of finite order and let $N \in \N$ be a divisor of $n$. Let $A$ be a support of $H$ realised by an automorphism $\phi_{H}$ of its underlying digraph $G_A$.   Then every point of $\xnN$ is on an orbit of length $N$ under the action of $H$ if and only if all circuits in $A$ are on orbits of length $N$ under the action of $\phi_{H}$.
\end{lemma}
\begin{proof}
    This proof follows straight-forwardly from the observation that the orbits of circuits in $A$ under the action of $\phi_{H}$ correspond to the action of $H$ on periodic points of $\xnN$. 
    Now as periodic points are dense in $\xnN$, the following chain of equivalences is true: all points of $\xnN$ are on orbits of length $N$ under the action of $H$ if and only if all periodic points of $\xnN$ are on orbits of length $N$ under the action of $H$ if and only if all circuits of $A$ are on orbits of length $N$ under the action of $H$.
 \end{proof}

\begin{lemma}\label{lemma:reverseimplicaton}
    Let $H \in \hn{n}$ be an element of finite  order and let $N$ be a divisor of $n$. Let $A$ be a support of $H$ realised by an automorphism $\phi_{H}$ of its underlying digraph $G_A$. Suppose that all circuits in $A$ are on orbits of length $N$, then $H$ is conjugate to a product of $n/N$ disjoint $N$-cycles.
 \end{lemma}
 \begin{proof}

 We may assume that $N>1$ since otherwise $H$ is the identity map.

By Corollary~\ref{cor:allstatesonlengthn} and Lemma~\ref{lemma:collapseadherence} we may assume that $A$ is an unimpeded support of $H$ and all states of $A$ have orbit length $N$ under $\phi_{H}$.

We consider two cases. 

First assume that there is only a single orbit of states under the action of $\phi_{H}$. As $A$ is unimpeded, the hypothesis of Lemma~\ref{Lem:relabelling one orbit} is satisfied for some $r$ dividing the orbit length of a state of $A$. We may then obtain an automaton $A'$ of size strictly smaller than $A$, supporting a conjugate $H'$ of $H$ realised by an automorphism $\phi_{H'}$ such that $A'$ has a single orbit of states and all edges of $A'$ have length $N$.

We then inductively apply Lemma~\ref{Lem:relabelling one orbit} until we have a single state automaton supporting a conjugate of $A$.

Now assume that there are multiple orbits of states under the action of $\phi_{H}$.

We claim that the are states $s,t \in Q_A$ belonging to distinct orbits under $\phi_{H}$ such that $\pi_{A}(x,s)$ and $\pi_{A}(x,t)$ belong to the  same orbit for all $x \in \xnp$.

To see this, consider a collapse sequence $A= A_0, A_1, A_2,\ldots,...$ where at each step we make the maximum possible collapses available subject to the constraint that the resulting induced partition of $Q_A$ has parts consisting only of states which belong to the same orbit. Since there are multiple orbits of states, there is some $m$ at which the sequence terminates (i.e no collapses are possible which satisfy the constraint) and $|A_m|>1$. Now since $A$ is synchronizing, there are states  $s,t \in Q_{A}$ belonging to distinct orbits such that $\pi_{A_m}(\cdot, [s])$ and $\pi_{A_m}(\cdot, [t])$ are identical maps on $\xnp$. It follows that $\pi_{A}(w, s)$ and $\pi_{A}(w,t)$ belong to the same orbit under the action of $\phi_{H}$ for all $w \in \xnp$.

Now we apply Lemma~\ref{lemma:rebraiding} and Corollary~\ref{cor:rebraidingreduction} to obtain a conjugate action of $H$ on a smaller support $A'$ which arises by identifying the pairs of states $s\phi_{\widehat{H}}^{i}, t\phi_{\widehat{H}}^{i}$ for all $i$. Let $\phi_{H'}$ be the induced action on $A'$.

We thus obtain a pair $(A,H')$ where all states of $A'$ have orbit length $N$ under the action $\phi_{H'}$ of $H'$ and $|A'|< |A|$.

We may therefore repeat the process until we have reduced to the first case.

 \end{proof}

Theorem~\ref{thm:mainresultIntro} now follows by putting together Lemmas~\ref{lemma:allorbitsonlegnthniffallcircuitsonlength} and \ref{lemma:reverseimplicaton}.

\begin{example}\label{exm:Hisconjugateto6cycle}
   We finish with a demonstration that the automaton $H$ in Figure~\ref{fig:ExampleACallup} is conjugate to a $6$-cycle.

    For this example it is enough to apply Lemma~\ref{Lemma:alledgesonlengthn} to obtain a support in which all edges have orbit length $6$. This is accomplished in Example~\ref{Example:allorbitsonlengthn} with resulting support $A'$ (depicted in Figure~\ref{fig:ExmaddingshaowstoB}) realised by an automorphism $\psi_{H}$ of $G_{A'}$.

    Now observe that the states $p_0$ and $q_0$ satisfy the hypothesis of Lemma~\ref{Lem:relabelling two orbit}. Applying Lemma~\ref{Lem:relabelling two orbit} is as in Example~\ref{Example:oneandtwoorbit}. Denote by $\psi_{H'}$ the induced conjugate action on $A'$ (with the new labels on its edges) and by $\phi_{H'}$ the corresponding action, after re-identifying shadow states, on $A$.  The resulting conjugate $H'$ of $H$ is the automaton to the left of Figure~\ref{fig:Exmoneround}. The conjugator, $I$ (depicted in Figure~\ref{fig:exmconjugator}) is the relabelling transducer.
Observe that the action of $\phi_{H'}$ on the outgoing edges of $p_i$ ($i =0,1,2$) in $A$ is mirrored on the outgoing edges of $q_i$. In particular for $i=0,1,2$, we may identify the pair of states $p_i, q_i$, to obtain an action $\varphi_{H'}$ on a smaller automaton $\bar{A}$. This is the automaton to the right of Figure~\ref{fig:Exmoneround}; the action of $\varphi_{H'}$ can be determined explicitly from $H'$. Note that the state $a_{2i + 1}$, $i = 0,1,2$, of $\bar{A}$ corresponds to the pair of states $q_{i}, p_{i}$ of $A$.

Notice that all edges of $G_{\bar{A}}$ are on orbits of length $6$ (under $\varphi_{H'}$) and $\bar{A}$ is unimpeded and has a single orbit of states. This is the regime of Lemma~\ref{Lem:relabelling one orbit} and the application of that lemma was worked out in Example~\ref{Example:oneandtwoorbit}.  The conjugator is the relabelling transducer $J$ in Figure~\ref{fig:Exmconjugatorround2}.%

The reader can verify that conjugating $H'$ by $J$ results in the single state transducer corresponding to the $6$-cycle $(0 \; 4 \; 2 \; 1 \; 5 \; 3)$.
Therefore, the element $IJ$ of $\hn{6}$ conjugates $H$ to the $6$-cycle $(0 \; 4 \; 2 \; 1 \; 5 \; 3)$.
$\bigcirc$
\end{example}

\printbibliography 
\end{document}